\documentclass[a4paper,11pt]{article}
\usepackage[english]{babel}
\usepackage[top=3cm, bottom=3cm, left=2.3cm, right=2.3cm]{geometry}

\usepackage{amssymb,amsthm,enumerate,xcolor,latexsym,dsfont}
\usepackage{bbm,mathrsfs}
\usepackage{amsmath}

\usepackage[pagebackref]{hyperref}

\newtheorem{theorem}{Theorem}[section]
\newtheorem{corollary}[theorem]{Corollary}
\newtheorem{definition}[theorem]{Definition}
\newtheorem{lemma}[theorem]{Lemma}
\newtheorem{remark}[theorem]{Remark}
\newtheorem{proposition}[theorem]{Proposition}

\newtheorem*{notation}{Notation}

\newcommand{\assign}{:=}
\newcommand{\backassign}{=:}

\newcommand{\dd}{\mathrm{d}}
\newcommand{\nin}{\not\in}
\newcommand{\tmscript}[1]{\text{\scriptsize{$#1$}}}
\newcommand{\bbE}{\mathbb{E}}

\newcommand{\bbR}{\mathbb{R}}
\newcommand{\bbX}{\mathbb{X}}

\newcommand{\cL}{\mathcal{L}}
\newcommand{\cP}{\mathcal{P}}

\newcommand{\bb}[1]{\mathbb{#1}}
\newcommand{\mc}[1]{\mathcal{#1}}

\newcommand{\cdummy}{\cdot}
\newcommand{\tmop}[1]{\ensuremath{\operatorname{#1}}}

\begin{document}

\title{Additive functionals as rough paths}

\author{Jean-Dominique Deuschel\thanks{We gratefully acknowledge financial support by the DFG via Research Unit FOR2402.}\\ TU Berlin \and
Tal Orenshtein\footnotemark[1]\hspace{4pt}\thanks{The main part of the work was carried out while T.O. was employed at HU Berlin and TU Berlin.}\\ TU Berlin \& WIAS \and
Nicolas Perkowski\footnotemark[1]\hspace{4pt}\thanks{The main part of the work was carried out while N.P. was employed at MPI MIS Leipzig. N.P. gratefully acknowledges financial support by the DFG via the Heisenberg program.}\\ FU Berlin}

\maketitle

\begin{abstract}
  We consider additive functionals of stationary Markov processes and show that under Kipnis-Varadhan type conditions they converge in rough path topology to a Stratonovich Brownian motion, with a correction to the L\'evy area that can be described in terms of the asymmetry (non-reversibility) of the underlying Markov process. We apply this abstract result to three model problems: First we study random walks with random conductances under the annealed law. If we consider the It\^o rough path, then we see a correction to the iterated integrals even though the underlying Markov process is reversible. If we consider the Stratonovich rough path, then there is no correction. The second example is a non-reversible Ornstein-Uhlenbeck process, while the last example is a diffusion in a periodic environment.

  As a technical step we prove an estimate for the $p$-variation of stochastic integrals with respect to martingales that can be viewed as an extension of the rough path Burkholder-Davis-Gundy inequality for local martingale rough paths of~\cite{Friz2008, Chevyrev2019, friz2018differential} to the case where only the integrator is a local martingale.
\end{abstract}


\section{Introduction}

In recent years there has been an increased interest in the link between homogenization and rough paths. It had been observed previously that homogenization often gives rise to non-standard rough path limits~\cite{lejay2003importance, Friz2015}. The more recent investigations were initiated with the work of Kelly and Melbourne~\cite{Kelly2017, Kelly2016Smooth, Kelly2016} who study rough path limits of additive functionals of the form $\sqrt n \int_0^{n\cdot} f(Y_s) \dd s$, where $Y$ is a deterministic dynamical system with suitable mixing conditions. In that way they are able to prove homogenization results for the convergence of deterministic multiscale systems of the type
\begin{align*}
	\dd X^n & = \sqrt{n} b(X^n_t, Y^n_t) \dd t,\\
	\dd Y^n & = n f(Y^n_t) \dd t,
\end{align*}
for which under suitable conditions $X^n$ converges to an autonomous stochastic differential equation. This line of research was picked up and extended for example by \cite{Bailleul2017, Chevyrev2019Multiscale, Chevyrev2019Deterministic,lopusanschi2018levy,lopusanschi2018ballistic}. More recent results also cover discontinuous limits~\cite{Chevyrev2019Superdiffusive}.

Motivated by this problem, as well as by the aim of understanding the invariance principle for random walks in random environment in rough path topology, we want to study rough path invariance principles for additive functionals $\sqrt n \int_0^{n\cdot} f(X_s) \dd s$ of Markov processes $X$ in generic situations. If we are only interested in a central limit theorem at a fixed time, then there are of course many results of this type and many ways of showing them. See for example~\cite{Peligrad2019} for a recent and fairly general result. A particularly successful approach for proving such a central limit theorem and even the functional central limit theorem (invariance principle) is based on Dynkin's formula and martingale arguments, and it was developed by Kipnis and Varadhan~\cite{Kipnis1986} for reversible Markov processes and later extended to many other situations; see the nice monograph~\cite{Komorowski2012}. Here we extend this approach to the rough path topology and we study some applications to model problems like random walks among random conductances, additive functionals of Ornstein-Uhlenbeck processes, and periodic diffusions.

This can also be seen as a complementary direction of research with respect to the recent advances in regularity structures~\cite{Bruned2019, Chandra2016, Bruned2017Renormalising}, where the aim is to find generic convergence results for models associated to singular stochastic PDEs. In those works the equations tend to be extremely complicated, but the approximation of the noise is typically quite simple (the prototypical example is just a mollification of the driving noise, but~\cite{Chandra2016} also allow some stationary mixing random fields that converge to the space-time white noise by the central limit theorem). In our setting the equation that we study is very simple (just a stochastic ODE), but the approximation of the noise is very complicated and (at least for us) it seems  difficult to check whether the conditions of~\cite[Theorem~2.34]{Chandra2016} are satisfied for the kind of examples that we are interested in.

The most interesting model that we study here is probably the random walk among random conductances. Here we distribute i.i.d. conductances $(\eta(\{x,y\}))_{x,y \in \bb Z^d: x \sim y}$ on the bonds of $\bb Z^d$ (where $x \sim y$ means $x$ and $y$ are neighbors). Then we let a continuous time random walk move along $\bb Z^d$, with jump rate $\eta(\{x,y\})$ from $x$ to $y$ (resp. from $y$ to $x$). We are interested in the large scale behavior, i.e. we study $n^{-1/2} X_{n t}$ for $n \to \infty$. It is well known that the path itself converges in distribution under the annealed law to a Brownian motion $B$ with an effective diffusion coefficient. Our contribution is to extend this convergence to the rough path topology, which allows us for example to understand the limit of discrete stochastic differential equations
\begin{equation}\label{eq:intro-sde}
	\dd Y^n_t = \sigma(Y^n_{t-}) \dd X^n_t,
\end{equation}
but also of SPDEs driven by $X^n$. And here we encounter a surprise: Even though $X$ is in a certain sense reversible (more precisely the underlying Markov process of the environment as seen from the walker is reversible), the iterated integrals $\int_0^\cdot X^n_s \otimes \dd X^n_s$ do not converge to $\int_0^\cdot B_s \otimes \dd B_s$, but instead we see a correction: We have
\[
	\left(X^n, \int_0^\cdot X^n_{s-} \otimes \dd X^n_s\right) \longrightarrow \left(B, \int_0^\cdot B_s \otimes \dd B_s + \Gamma t\right),
\]
where $\Gamma$ is a correction given by
\[
	\Gamma = \frac{1}{2} \langle B, B \rangle_1 - E_{\pi} [\eta (\{0, e_1\})] I_d = \frac12 \left( \langle B, B \rangle_1 - E_{\pi} [\eta (\{0, e_1\}) + \eta(\{0,-e_1\})] I_d\right)
\]
for the law $\pi$ of the random conductances. Of course, $\Gamma$ vanishes if the conductances are deterministic (i.e. if $\pi$ is a Dirac measure). But if the conductances are truly random, then typically the effective diffusion is not just given by the expected conductance, and in $d=1$ one can even show that this is never the case (see the discussion at the top of p.89 of \cite{Komorowski2012}). Therefore, $\Gamma$ is typically non-zero, and the solution $Y^n$ of \eqref{eq:intro-sde} converges to the solution $Y$ of
\[
	\dd Y_t = \sigma(Y_t) \dd B_t + \sum_{j,k,\ell} \partial_k \sigma_{\cdot j}(Y_t) \sigma_{k\ell}(Y_t) \Gamma_{j \ell} \dd t.
\]
If on the other hand we denote by $\tilde X^n$ the linear interpolation of the pure jump path $X^n$, then $(\tilde X^n, \int_0\tilde X^n_s \dd \tilde X^n_s)$  converges to the limit that we would naively expect, namely to the Stratonovich rough path above $B$. From the point of view of stochastic calculus this is a bit surprising: After all, there are stability results for It\^o integrals \cite{Kurtz1991}, while the quadratic variation (i.e. the difference between It\^o and Stratonovich integrals) is very unstable. In fact, we are not aware of any previous results of this type (naive limit for the Stratonovich rough path, correction for the It\^o rough path), but it seems to be a generic phenomenon. The same effect appears for periodic diffusions, and we expect to see it for nearly all models treated in the monograph~\cite{Komorowski2012}.
On the other hand, for ballistic random walks in random environment, after centering, a correction to the \emph{Stratonovich} rough path is identified in terms of the expected stochastic area on a regeneration interval \cite[Theorem~3.3]{lopusanschi2018ballistic}. Moreover, for random walks in deterministic periodic environments simple examples for non-vanishing corrections are available \cite[Section 1.2]{lopusanschi2018levy} (or \cite[Section~4.2]{lopusanschi2018ballistic}).
For processes that can be handled with the Kipnis-Varadhan approach we generically expect to see a correction to the Stratonovich rough path if and only if the underlying Markov process is non-reversible.

\paragraph{Structure of the paper} In the next section we introduce some basic notions from rough path theory. Section~\ref{sec:kipnis-varadhan} presents our main result Theorem~\ref{thm:additive-rp}, the rough path invariance principle for additive functionals of stationary Markov processes, which holds under the same conditions as the abstract result in~\cite{Komorowski2012}. The proof is based on recent advances on It\^o rough paths with jumps due to Friz and Zhang~\cite{friz2018differential}, on stability results for It\^o integrals under the so called UCV condition by Kurtz and Protter~\cite{Kurtz1991}, on L\'epingle's Burkholder-Davis-Gundy inequality in $p$-variation~\cite{Lepingle1975}, and on repeated integrations by parts together with a new estimate on the $p$-variation of stochastic integrals (Proposition~\ref{prop:area-variation}). In Section~\ref{sec:applications} we apply our abstract result to three model problems: random walks among random conductances, additive functionals of Ornstein-Uhlenbeck processes, and periodic diffusions. Finally, Section~\ref{sec:proof-area-variation} contains the proof of Proposition~\ref{prop:area-variation} which might be of independent interest.

\paragraph{Notation} For two families $(a_i)_{i \in I},(b_i)_{i\in I}$ of real numbers indexed by $I$ the notation $a_i\lesssim b_i$ means that $a_i\le c\, b_i$ for every $i\in I$ where $c\in(0,\infty)$ is a constant. Let $\Delta_T:=\{s,t\in[0,T]: s\le t\}$ for $T>0$. We interpret any function $X:[0,T]\to\bbR^d$ also as a function on $\Delta_T$ via $X_{s, t}:=X_t-X_s$, $(s,t)\in\Delta_T$.
{For a metric space $(E,d)$ we write $C([0,T], E)$ resp. $D([0,T], E)$ for the continuous resp. c\`adl\`ag functions from $[0,T]$ to $E$. A function $X:\Delta_T \to E$ is called continuous resp. c\`adl\`ag if for all $s \in [0,T)$ the map $t\mapsto X_{s,t}$ on $[s,T]$ is continuous resp. c\`adl\`ag, and we write $C(\Delta_T, E)$ resp. $D(\Delta_T,E)$ for the corresponding function spaces.}

\section{Elements of rough path theory}

Here we recall some basic elements of rough path theory for It\^o rough paths with jumps. See~\cite{friz2018differential} for much more detail.

Let us write $\|X\|_{\infty, [0, T]} : = \sup_{t \in [0, T]} |X_t|$ (resp. $\|\bbX\|_{\infty, [0, T]} : = \sup_{(s,t) \in \Delta_T} |\bbX_{s,t}|$) to denote the uniform norm of $X \in  D([0, T], \bbR^d)$ (resp. $\bbX \in D(\Delta_T, \bbR^{d \times d})$). For $0<p<\infty$ and a normed space $(E,|\cdot|_{E})$, we define the $p$-variation of {$\Xi : \Delta_T \to E$ (and so in particular of $\Xi : [0,T]\to E$)}  by
\begin{align}
  \|\Xi\|_{p, [0,T]}
  \;: =\;
  \bigg(
    \sup_{\cP} \sum_{[s,t] \in \cP} |\Xi_{s,t}|^p
  \bigg)^{\!\!1/p}
  \;\in\;
  [0, +\infty],
\end{align}
where the supremum is taken over all finite partitions $\cP$ of $[0,T]$ and the summation is over all intervals $[s,t] \in \cP$.  Note that for any $0 < p \leq q < \infty$, we have that $\|\Xi\|_{q, [0,T]} \leq \|\Xi\|_{p,[0,T]}$.

\begin{definition} [$p$-variation rough path space]
For $p \in [2,3)$, the space $D_{p\text{-}\mathrm{var}}([0,T], \bbR^d\times \bbR^{d \times d})$ (resp. $C_{p\text{-}\mathrm{var}}([0,T], \bbR^d\times \bbR^{d \times d})$) of c\`adl\`ag (resp. continuous) $p$-variation rough paths is defined by the subspace of all functions $(X, \bbX) \in D([0,T], \bbR^d) \times D(\Delta_T, \bbR^{d \times d})$ satisfying Chen's relation, that is,
\begin{equation}\label{eq:Chen}
\bbX_{r,t} - \bbX_{r,s} - \bbX_{s,t} = X_{r,s} \otimes X_{r,t}
\end{equation}
for $0 \leq r \leq s \leq t \leq T$, and
\begin{align}
  |X_0| \,+\,
  \|X\|_{p, [0,T]}
  \,+\, \|\bbX\|_{p/2, [0,T]}^{1/2}
  \;<\;
  \infty.
\end{align}
\end{definition}
{The $p$-variation Skorohod distance on $D_{p\text{-}\mathrm{var}}([0,T], (\bbR^d,\bbR^{d \times d}))$ is
\[
	\sigma_{p,[0,T]}((X,\mathbb X), (Y, \bb Y)) := \inf_{\lambda \in \Lambda_T} \{|\lambda| \vee \big( \|X\circ \lambda - Y \circ \lambda\|_{p,[0,T]} + \|\bb X\circ(\lambda,\lambda) - \bb Y\circ(\lambda,\lambda)\|_{p/2,[0,T]} \big) \},
\]
where $\Lambda_T$ are the strictly increasing bijective functions from $[0,T]$ onto itself, and $|\lambda| = \sup_{t \in [0,T]} |\lambda(t) - t|$. The uniform Skorohod distance is defined similarly, except with the $p$-variation respectively $p/2$-variation distance replaced by the uniform distance; see \cite[Section~5]{friz2018differential} for details.
}

For $X,Y\in D([0,T],\bbR^d)$ we use the notation $\int_0^t Y_{s-} \otimes \dd X_s$ for the left-point Riemann integral, that is
\[
\int_0^t Y_{s-} \otimes \dd X_s := {\int_{(0,t]} Y_{s-}\otimes \dd X_s := }
{\lim_{n \to \infty}\big\{\sum_{[u,v]\in\cP_n} Y_{u} \otimes \big(X_{v} - X_{u}\big)\big\},}
\]
whenever this limit is well defined along an {implicitly fixed sequence of partitions $(\cP_n)$ of $[0,t]$ with mesh size going to zero.}
Note that if $X$ is a semimartingale and $Y$ is adapted to the same filtration, then this definition coincides with the Ito integral.
We remark also that the iterated integrals
\[
\bbX_{s,t}:=\int_s^t X_{s,u-} \otimes \dd X_u = \int_{(s,t]} X_{s,u-}\otimes \dd X_u,
\]
satisfy Chen's relation \eqref{eq:Chen}. Moreover, so do $\tilde \bbX_{s,t} := \bbX_{s,t} + (t-s) \Gamma$, for any fixed matrix $\Gamma$.

\begin{remark}
  \label{rmk:one-parameter}Note that by Chen's relation $\mathbb{X}_{s, t}
  =\mathbb{X}_{0, t} -\mathbb{X}_{0, s} - X_{0, s} \otimes X_{s, t}$ whenever
  $0 \leqslant s \leqslant t \leqslant T$, and therefore
\begin{align*}
    \| \mathbb{X}-\mathbb{Y} \|_{\infty, [0, T]} & = \sup_{0 \leqslant s < t
    \leqslant T} | \mathbb{X}_{s, t} -\mathbb{Y}_{s, t} |\\
    & \lesssim \| \mathbb{X}_{0, \cdot} -\mathbb{Y}_{0, \cdot}
    \|_{\infty, [0, T]} + (\| X_{0, \cdot} \|_{\infty, [0, T]} \vee \| Y_{0,
    \cdot} \|_{\infty, [0, T]}) \| X_{0, \cdot} - Y_{0, \cdot}
    \|_{\infty, [0, T]} .
  \end{align*}
  Consequently, the uniform resp. Skorohod distance of the (one-parameter) paths $(X_{0, \cdot},
  \mathbb{X}^n_{0, \cdot})$ and $(Y_{0,\cdot}, \bb Y_{0,\cdot})$ controls the uniform resp. Skorohod distance of $(X,\bb X)$ and $(Y,\bb Y)$.
\end{remark}

The following lemma by \cite{friz2018differential} will be useful in the sequel.
\begin{lemma}  \label{lem:rp-conv}
  Let $ (Z^n, \mathbb{Z}^n)$ be a sequence of c\`adl\`ag rough paths and let $p<3$. Assume that
  there exists a c\`adl\`ag rough path $(Z, \mathbb{Z})$ such that $(Z^n,
  \mathbb{Z}^n) \rightarrow (Z, \mathbb{Z})$ in distribution in the Skorohod (resp. uniform) topology and that the family of real valued
  random variables $(\| (Z^n, \mathbb{Z}^n) \|_{p, [0, T]})_n$ is tight. Then
  $(Z^n, \mathbb{Z}^n) \rightarrow (Z, \mathbb{Z})$ in distribution in the
  $p'$-variation Skorohod (resp. uniform) topology for all $p' \in (p, 3)$.
\end{lemma}
\begin{proof} This follows from a simple interpolation argument, see the proof of Theorem 6.1 in \cite{friz2018differential}.
\end{proof}

Invariance principles for rough path sequences guarantee the convergence of the solutions to rough differential equations where the noise is approximated by the path sequence. Moreover, whenever the second level (the first order `iterated integrals') of the rough path has a correction, the limiting path solves a drift-modified rough equation defined explicitly in terms of the correction. More precisely, \cite[Theorem 6.1 and Proposition 6.9]{friz2018differential} proved the following.
  {Let $(Z^n)$ be a sequence of semimartingales and assume that $(Z^n, \int_0^\cdot Z^n_{s-} \dd Z^n_s)$ converges in distribution in $p$-variation Skorohod (resp. uniform) distance to a rough path $(Z, \bb Z)$, where $Z$ is a semimartingale and $\mathbb{Z}_{s, t} = \int_s^t Z_{s, r -} \otimes
  \dd Z_r + \Gamma \times(t - s)$ for $\Gamma \in \mathbb{R}^{d \times d}$. Then the solutions $(Y^n)$ of
    \[ \dd Y_t^n = \sigma (Y^n_{t -}) \dd Z_t^n,\qquad Y^n_0 = y, \]
    converge in distribution in the Skorohod (resp. uniform) topology to the solution $Y$ of
    \[
    	\dd Y_t  = \sigma(Y_{t-}) \dd (Z, \bb Z)_t = \sigma(Y_{t-})\dd Z_t + \sum_{j,k,\ell} \partial_k \sigma_{\cdot j}(Y_t) \sigma_{k\ell}(Y_t) \Gamma_{j \ell} \dd t,\qquad Y_0 = y,
    \]
    where $\dd(Z,\bb Z)_t$ denotes rough path integration and $\dd Z_t$ is just the It\^o integral.}

\section{Additive functionals as rough paths}\label{sec:kipnis-varadhan}

Here we present our abstract convergence result for additive functionals of stationary Markov processes. We place ourselves in the context of Chapter~2 in {\cite{Komorowski2012}}: Let
$(X_t)_{t \geqslant 0}$ be a c{\`a}dl{\`a}g Markov process in a filtration
satisfying the usual conditions, with values in a Polish space $E$, and let
$\pi$ be a stationary probability measure for $X$ and $X_0 \sim \pi$. We assume that the
transition semigroup of $X$ can be extended to a strongly continuous
contraction semigroup $(T_t)_{t \geqslant 0}$ on $L^2 (\pi)$. We write
$\mathcal{L}$ for the infinitesimal generator of $(T_t)$ and we assume that
$\pi$ is ergodic for $\mathcal{L}$, i.e. that $F$ is $\pi$-almost surely
constant whenever $\pi (\{\mathcal{L}F = 0\}) = 1$. We also assume that there
exists a common core $\mathcal{C}$ for $\mathcal{L}$ and $\mathcal{L}^{\ast}$,
where $\mathcal{L}^{\ast}$ is the $L^2 (\pi)$-adjoint of $\mathcal{L}$, and
that $\mathcal{C}$ contains the constant functions. We write
\[
	\mc L_S = \frac12(\mc L + \mc L^\ast)\qquad \text{and} \qquad \mc L_A = \frac12 (\mc L - \mc L^\ast).
\]

\begin{notation}
  We write $\mathbb{P}$ or $\mathbb{P}_{\pi}$ (and $\mathbb{E}$ or
  $\mathbb{E}_{\pi}$) for the distribution of the stationary process $(X_t)_{t
  \geqslant 0}$ on the Skorohod space $D (\mathbb{R}_+, E)$. The notation
  $E_{\pi}$ is reserved for the integration with respect to $\pi$ on the space $E$.
\end{notation}

\begin{definition}
  The space $\mathcal{H}^1$ is defined as the completion of $\mathcal{C}$ with
  respect to the norm
  \[
  \| F \|_1^2 \assign E_{\pi} [F (-\mathcal{L}) F] = E_{\pi} [F (-\mathcal{L}_S) F],
  \]
  {or more precisely we identify $F, G \in \mc C$ if $\|F - G\|_1 = 0$, and $\mc H^1$ is the completion of the equivalence classes.}
  The space $\mathcal{H}^{- 1}$ is the dual of $\mathcal{H}^1$: We define for $F \in \mc C$
  \[
  \| F \|_{- 1}^2 \assign \sup_{\substack{ G \in \mathcal{C}:\\ \| G \|_1 \leqslant 1}} E_{\pi} [F G] = \sup_{G \in \mathcal{C}} \{ 2 E_{\pi} [F G]
     - \| G \|_1^2 \} \]
  and then $\mathcal{H}^{- 1}$ is the completion of $\{ F \in \mathcal{C}: \|
  F \|_{- 1} < \infty \}$ with respect to $\| \cdot \|_{- 1}$.

  {If $F$ takes values in $\bb R^d$ we also write $F \in \mc H^1(\bb R^d)$, $F \in \mc H^{-1}(\bb R^d)$ or $F \in L^2(\pi, \bb R^d)$, etc.}
\end{definition}

Note that if $E_{\pi} [F] \neq 0$ then we can take $G = \lambda \in
\mathcal{C}$ for $\lambda \in \mathbb{R}$ so that $\| G \|_1 = 0$ and by
sending $\lambda \rightarrow \pm \infty$ we see that $\| F \|_{- 1} = \infty$.
Therefore, we get that $E_{\pi} [F] = 0$ for all $F \in L^2 (\pi) \cap
\mathcal{H}^{- 1}$.

Let now $F \in \mathcal{H}^{- 1} (\mathbb{R}^d) \cap L^2 (\pi, \mathbb{R}^d)$.
Our aim is to derive a scaling limit for the (absolutely continuous) rough path $(Z^n,
\mathbb{Z}^n)$, where
\[ Z^n_{s, t} \assign \frac{1}{\sqrt{n}} \int_{n s}^{n t} F (X_r) \dd r,
   \qquad \mathbb{Z}^n_{s, t} \assign \int_s^t Z^n_{s, r} \otimes \dd Z^n_r,
   \]
and the integration with respect to $Z^n_r$ is in the Riemann-Stieltjes sense.
Let us first recall the following result:

\begin{lemma}[\cite{Komorowski2012}, Theorem~2.33]
  \label{lem:uniform-tightness-for-path}

  Assume that $\pi$ is ergodic for $\mathcal{L}^{\ast}$. Let $F \in L^2 (\pi,
  \mathbb{R}^d) \cap \mathcal{H}^{- 1} (\mathbb{R}^d)$ and assume that the
  solution $\Phi_\lambda$ to the resolvent equation $(\lambda -\mathcal{L}) \Phi_{\lambda} = F$ with $\lambda > 0$ satisfies
  \begin{equation}
    \label{eq:additive-rp-assumption-1} \lim_{\lambda \rightarrow 0} \left(
    \sqrt{\lambda} \| \Phi_{\lambda} \|_{\pi} + \| \Phi_{\lambda} - \Phi \|_1
    \right) = 0
  \end{equation}
  for some $\Phi \in \mathcal{H}^1(\bb R^d)$. Then $(Z^n)_n$ converges in distribution
  in $C ([0,T], \mathbb{R}^d)$ to a Brownian motion with covariance
  matrix
  \[ \langle B, B \rangle_t = 2 t \langle \Phi, \otimes \Phi \rangle_1 \assign
     2 t (\langle \Phi^{(k)}, \Phi^{(\ell)} \rangle_1)_{1 \leqslant k, \ell
     \leqslant d} = 2 t \lim_{\lambda \rightarrow 0} (\langle
     \Phi_{\lambda}^{(k)}, \Phi_{\lambda}^{(\ell)} \rangle_1)_{1 \leqslant k,
     \ell \leqslant d} . \]
\end{lemma}

Our aim is to extend Lemma~\ref{lem:uniform-tightness-for-path} to the rough
path topology.
Our main result is:

\begin{theorem}\label{thm:additive-rp}
Let $p > 2$. Under the assumptions of
  Lemma~\ref{lem:uniform-tightness-for-path} the process $(Z^n, \mathbb{Z}^n)$
  converges weakly to
  \begin{equation}\label{eq:limit of af}
    \left( B_t, {\int_0^t B_s \otimes \circ \dd B_s + \Gamma t} \right)_{t \geqslant 0}
  \end{equation}
  in the (uniform) $p$-variation topology on $C_{p\text{-}\mathrm{var}}([0,T], \bbR^d\times \bbR^{d \times d})$, where $B$ is the same Brownian motion as in
  Lemma~\ref{lem:uniform-tightness-for-path}, {
  \[
  	\int_0^t B_s \otimes \circ \dd B_s := \int_0^t B_s \otimes \dd B_s + \frac12 \langle B, B\rangle_t
  \]
  }
  denotes Stratonovich integration, and {$\Gamma$ is given by the following limit, which exists:}
  \[ \Gamma = \lim_{\lambda \rightarrow 0} E_{\pi} [\Phi_{\lambda} \otimes
     \mathcal{L}_A \Phi_{\lambda}].
  \]
\end{theorem}
For the rest of the section we shall assume without further mention that the conditions of Theorem~\ref{thm:additive-rp} are satisfied.

\begin{remark}
  As $Z^n$ is of finite variation the iterated integral $\int_0^t
  Z^n_s \otimes \dd Z^n_s$ ``wants'' to converge to the Stratonovich
  integral, and $\Gamma$ describes the area correction. Note that $\Gamma = 0$ if $\mathcal{L}$ is symmetric, i.e. if $X$
  is reversible, so in that case we indeed obtain the Stratonovich rough path
  over $B$.
\end{remark}

\begin{remark}
  In Lemma~\ref{lem:uniform-tightness-for-path} and
  Theorem~\ref{thm:additive-rp} the ergodicity of $\pi$ with respect to
  $\mathcal{L}^{\ast}$ is only needed for proving the tightness of $(Z^n,
  \mathbb{Z}^n)$ in the uniform topology. This is relatively subtle because we need tightness of certain martingales $M^\Psi$ for which we only know that $\mathbb{E} [\langle M^\Psi
  \rangle_t - \langle M^\Psi \rangle_s] \lesssim |t - s|$, which is insufficient to apply Kolmogorov's continuity criterion. If we can show
  $\mathbb{E} [| \langle M^\Psi \rangle_t - \langle M^\Psi \rangle_s |^{1 + \delta}]
  \lesssim |t - s|^{1 + \delta}$ for some $\delta > 0$ and for the
  martingales $M^\Psi$ of Lemma~\ref{lem:fb-decomp} below, then we do not need the
  ergodicity of $\pi$ with respect to $\mathcal{L}^{\ast}$ (although we do
  need ergodicity with respect to $\mathcal{L}$).
\end{remark}

The strategy for proving Theorem \ref{thm:additive-rp} is to apply Lemma~\ref{lem:rp-conv}, which
separates the convergence proof into two problems: \ Showing tightness of $
(Z^n, \mathbb{Z}^n)$ in the $p$-variation topology, and identifying
the limit in the Skorohod topology.
To identify the limit we follow a similar strategy as in \cite{Komorowski2012} and combine it with tools from rough paths together with a simple integration by parts formula.

{Let us formally sketch how the correction $\Gamma$ arises, under the assumption that we can solve the Poisson equation $-\mc L \Phi = F$ (which is for example the case if $X$ has a spectral gap and $E_\pi[F] = 0$). In that case we have
\[
	Z^n_t = \frac{1}{\sqrt{n}} \Phi(X_0) - \frac{1}{\sqrt{n}} \Phi(X_{n t}) + M^{\Phi,n}_t
\]
for a sequence of martingales $(M^{n})$. Therefore,
\[
	\int_0^t Z^n_s \otimes \dd Z^n_s =  \int_0^t (\Phi(X_0) - \Phi(X_{n s})) \otimes F(X_{ns}) \dd s + \int_0^t M^n_s \otimes \dd Z^n_s.
\]
By the ergodic theorem, the first term on the right hand side converges to
\[
	t ( \Phi(X_0)  \otimes E_\pi[F] - E_\pi[\Phi \otimes F ] ) = t E_\pi [ \Phi \otimes \mc L \Phi].
\]
To understand the contribution of the remaining contribution we use integration by parts: Since $Z^n$ is of finite variation, we have
\[
	\int_0^t M^n_s \otimes \dd Z^n_s = M^n_t \otimes Z^n_t - \int_0^t Z^n_s \otimes \dd M^n_s.
\]
Since $X$ is stationary, we have $Z^n = M^n + O(n^{-1/2})$, and $(Z^n, M^n)$ converges jointly to $(B,B)$. And the martingale sequence $(M^n)$ satisfies the ``UCV condition'' (see the next section for details), and therefore $\int_0^t Z^n_s \otimes \dd M^n_s \to \int_0^t B_s \otimes \dd B_s$. After passing to the limit we apply integration by parts once more and deduce that
\[
	\int_0^t M^n_s \otimes \dd Z^n_s \to  B_t \otimes B_t - \int_0^t B_s \otimes \dd B_s = \int_0^t B_s \otimes \dd B_s + \langle B, B\rangle_t.
\]
So overall
\begin{align*}
	\int_0^t Z^n_s \otimes \dd Z^n_s  & \to \int_0^t B_s \otimes \circ \dd B_s + \frac12 \langle B, B \rangle_t + t E_\pi [ \Phi \otimes \mc L \Phi]\\
	& = \int_0^t B_s \otimes \circ \dd B_s + t E_\pi[\Phi \otimes (-\mc L_S) \Phi] + t E_\pi [ \Phi \otimes \mc L \Phi] \\
	& = \int_0^t B_s \otimes \circ \dd B_s + t E_\pi[\Phi \otimes \mc L_A \Phi]
\end{align*}
}

\subsection{Tightness in $p$-variation}

The following notion was introduced by Kurtz-Protter \cite{Kurtz1991}.

{
\begin{definition}[UCV condition]\label{def:UCV}
Let $(X^n)_{n\geqslant 1}\subset D([0,T],\bbR)$ 
be a sequence of c\`adl\`ag local martingales. We say that $(X^n)_{n\geqslant 0}$ satisfies the
\emph{Uniformly Controlled Variation} (UCV) condition if
\[
\sup_{n}\bbE
\left[ [X^{n}]_{T} \right]  < \infty.
\]
\end{definition}
Strictly speaking this is a very particular special case of the definition by Kurtz and Protter, who are much more permissive and consider general semimartingales rather than local martingales, and they allow for localization with stopping times as well as truncation of large jumps. But here we only need the special case above.
}

The celebrated result of Kurtz-Protter \cite{Kurtz1991} guarantees the convergence in the Skorokhod topology of the stochastic integrals of a sequence of c\`adl\`ag local martingales satisfying the UCV condition.
{Before we state it, recall that a sequence of processes $(Y^n)_{n \in \mathbb{N}}$ in $D (\mathbb{R}_+,
\mathbb{R}^d)$ is called \emph{C-tight} if it is tight in the Skorohod
topology and all limit points are continuous processes.}

\begin{theorem}[\cite{Kurtz1991}, Theorem 2.2]\label{lem:UCV}
Let $(X^n, Y^n)_{n\geqslant 1}\subset D([0,T],\bbR^2)$ be converging in probability in the Skorokhod topology
(or jointly in distribution) to a pair $(X,Y)\in D([0,T],\bbR^2)$. Suppose that
$(X^n)_{n\geqslant 1}$ is a sequence of local martingales which satisfies the UCV condition.
Then,
$ (X^n,Y^n,\int_0^\cdot Y^n_{s-}  \dd X^n_s)$
converges to
$(X,Y,\int_0^\cdot Y_{s-}  \dd X_s)$
as $n\to\infty$ in probability (or weakly) in $D([0,T],\bbR^3)$.
In particular, if in addition $\int_0^\cdot Y_{s-}  \dd X_s\in C([0,T],\bbR)$,
then $\int_0^\cdot Y^n_{s-}  \dd X^n_s$ is $C$-tight.
\end{theorem}

\begin{corollary}\label{cor:UCV}
Let $(X^n, Y^n)_{n\geqslant 1}\subset D([0,T],\bbR^2)$ satisfy the same assumptions as in Theorem~\ref{lem:UCV}.
If in addition $(Y_n)_{n\ge1}$ is a sequence of semimartingales and $(X^n, Y^n, [X^n,Y^n])$ converges to $(X,Y,A)$ in probability (or jointly in distribution), where $Y$ is a semimartingale and $A$ is an adapted c\`adl\`ag process of finite variation,
then
\[
	\left(X^n, Y^n, \int_0^\cdot X^n_{s-} \otimes \dd Y^n_s,  \int_0^\cdot Y^n_{s-} \otimes \dd X^n_s\right) \longrightarrow \left(X, Y, \int_0^\cdot X_{s-} \otimes \dd Y_s + [X,Y] - A,  \int_0^\cdot Y_{s-} \otimes \dd X_s \right)
\]
in probability (or weakly) in $D([0,T],\bbR^4)$.
In particular, if in addition $\int_0^\cdot X_{s-} \otimes \dd Y_s\in C([0,T],\bbR)$, then $\int_0^\cdot X^n_{s-} \otimes \dd Y^n_s$ is $C$-tight.
\end{corollary}
\begin{proof}
Using integration by parts, we have
\[
	\int_0^\cdot X^n_{s-} \otimes \dd Y^n_s = X^n \otimes Y^n - X^n_0 \otimes Y^n_0 -  \int_0^\cdot Y^n_{s-} \otimes \dd X^n_s - [X^n, Y^n],
\]
so that the claim follows from the Kurz-Protter result together with another integration by parts:
\[
	X \otimes Y - X_0 \otimes Y_0 -  \int_0^\cdot Y_{s-} \otimes \dd X_s - A = \int_0^\cdot X_{s-} \otimes \dd Y_s + [X,Y] - A.
\]
\end{proof}

Throughout this section we will often use the following representation of
additive functionals:

\begin{lemma}
  \label{lem:fb-decomp}Let $\Psi \in
  \mathcal{C} (\mathbb{R}^m)$. Then we have for $T > 0$
  \begin{equation}
    \label{eq:fb-decomp}  \int_0^t \mathcal{L}_S \Psi (X_s) \dd s = \frac{1}{2} (M^\Psi_t + \hat{M}^\Psi_T -
    \hat{M}^\Psi_{T - t}), \qquad t \in [0, T]
  \end{equation}
  where $M^\Psi$ is a martingale and $\hat{M}^\Psi$ is a martingale with respect to
  the backward filtration $\hat{\mathcal{F}}_t = \sigma (X_{T - s} : s
  \leqslant t)$, such that
  \begin{equation}
    \label{eq:fb-decomp-qv} \mathbb{E} [\langle M^\Psi \rangle_t] =\mathbb{E}
    [\langle \hat{M}^\Psi \rangle_t] = 2 E_{\pi} [\Psi \otimes (-\mathcal{L}_S)
    \Psi] t = 2 t (\langle \Psi_k, \Psi_{\ell} \rangle_1)_{1 \leqslant k,
    \ell \leqslant m}, \qquad t \in [0, T] .
  \end{equation}
  Assume that $\pi$ is ergodic for $\mathcal{L}^{\ast}$. Then under the
  rescaling $T \rightarrow n T$ and $M^{\Psi, n}_t = n^{- 1 / 2} M^\Psi_{n t}$ and
  similarly for $\hat{M}^{\Psi, n}$ both processes converge in distribution in $D
  ([0, T], \mathbb{R}^m)$ to a Wiener process, and by
  {\eqref{eq:fb-decomp-qv}} they satisfy the UCV condition.
  If $G,H \in L^2(\pi, \mathbb{R}^m)$ and
  $A_{s, t} = \int_s^t \int_s^{r_1} G (X_{r_2})\dd r_2 \otimes H (X_{r_1}) \dd r_1$ for $0 \leqslant s < t \leqslant
  T$, then
\begin{align} \label{eq:fb-decom-area}
    A_{s, t} & = \frac{1}{2} \int_s^t \int_s^{r_1} G (X_{r_2}) \dd r_2
    \otimes \dd M^\Psi_{r_1} - \frac{1}{2} \int_{T - t}^{T - s} \int_{T -
    r_1}^t G (X_{r_2}) \dd r_2 \otimes \dd \hat{M}^\Psi_{r_1}\\ \nonumber
    & \quad + \frac{1}{2} \int_s^t G (X_r) \dd r \otimes (\hat{M}^\Psi_{T -
    s} - \hat{M}^\Psi_{T - t}) + \int_s^t \int_s^{r_1} G (X_{r_2}) \dd r_2
    \otimes (H (X_{r_1}) -\mathcal{L}_S \Psi (X_{r_1})) \dd r_1 .
  \end{align}
\end{lemma}

\begin{proof}
  The representation {\eqref{eq:fb-decomp}} is obtained e.g.\
   by applying Dynkin's formula to $\Psi (X)$ and $\Psi (\hat X)$ on $[0,u], u\in[0,T]$, and then computing $M^\Psi_t + \hat{M}^\Psi_T -
    \hat{M}^\Psi_{T - t}$. {If $\Psi^2$ is in the domain of $\mc L$, then also~{\eqref{eq:fb-decomp-qv}} follows from Dynkin's formula; otherwise we use an approximation argument, see p.35 of \cite{Komorowski2012}.} For the convergence
  of $M^{\Psi, n}$ and $\hat{M}^{\Psi, n}$ see the proof of Theorem~2.32/2.33 in
  {\cite{Komorowski2012}}. The representation for $A_{s, t}$ follows by
  writing the integral against $\hat{M}^\Psi_{T - \cdot}$ as a limit of Riemann
  sums -- note that $\int_0^{\cdot} G (X_r) \dd r$ is continuous and of
  finite variation, so the integral is defined pathwise and we do not need to
  worry about quadratic covariations or the difference between forward and
  backward integral.
\end{proof}
For $f \in D (\mathbb{R}_+, \mathbb{R}^d)$ and $\delta, T > 0$ we define the
modulus of continuity
\[ w_T (f, \delta) \assign \sup_{\tmscript{\begin{array}{c}
     s, t \in [0, T] :\\
     | s - t | \leqslant \delta
   \end{array}}} | f (t) - f (s) | . \]
We will need the
following lemma:

\begin{lemma}[{\cite{Jacod2003}}, Proposition VI.3.26]
  \label{lem:C-tight}

  The sequence $(Y^n)$ is C-tight if and only if the following two conditions
  hold:
  \begin{enumerate}
    \item[i.] For all $T > 0$ we have
    \[ \lim_{K \rightarrow \infty} \limsup_{n \rightarrow \infty} \mathbb{P}
       (\sup_{t \in [0, T]} | Y^n_t | > K) = 0 ; \]
    \item[ii.] for all $\varepsilon, T > 0$ we have
    \[ \lim_{\delta \rightarrow 0} \limsup_{n \rightarrow \infty} \mathbb{P} (w_T
       (Y^n, \delta) > \varepsilon) = 0. \]
  \end{enumerate}
  If $(Y^n)$ is a sequence of processes in $C (\mathbb{R}_+, \mathbb{R}^d)$,
  then these two conditions are equivalent to tightness in the uniform
  topology.
\end{lemma}

Since the uniform modulus of continuity is subadditive, i.e. $w_T (f + g, \delta)
\leqslant w_T (f, \delta) + w_T (g, \delta)$, it follows from this Lemma that the sum of
two C-tight sequences is again C-tight. Note that the same is not necessarily true for
sequences that are tight in the Skorohod topology on $D (\mathbb{R}_+,
\mathbb{R}^d)$.
\begin{lemma}
  \label{lem:uniform-tightness}Under the assumptions of
  Lemma~\ref{lem:uniform-tightness-for-path} the sequence $(Z^n,
  \mathbb{Z}^n)$ is tight in $C (\mathbb{R}_+, \mathbb{R}^d \oplus
  \mathbb{R}^{d \otimes d})$.
\end{lemma}

\begin{proof}
  By Lemma~\ref{lem:uniform-tightness-for-path} and
  Remark~\ref{rmk:one-parameter} it suffices to show that $\mathbb{Z}^n_{0,
  \cdot}$ is tight in $C (\mathbb{R}_{_{} +}, \mathbb{R}^{d \otimes d})$.
  We shall use Lemma \ref{lem:C-tight}. Since the set $\mathcal{L}_S \mathcal{C}$ is dense in $\mathcal{H}^{- 1}$, see
  Claim B on p.42 of {\cite{Komorowski2012}}, we can find
  $\Phi_m \in \mathcal{C} (\mathbb{R}^d)$ so that
  $\| F -\mathcal{L}_S \Phi_m \|_{- 1} < \frac{1}{m}$.
  Then eq.~\eqref{eq:fb-decom-area} from Lemma \ref{lem:fb-decomp} {(or rather a slight modification with the inner integral in the second term on the right hand side running from $T-r_1$ to $T$ instead of from $T-r_1$ to $t$)}
  gives
\begin{align*}
    \mathbb{Z}^n_{0, t} & = \frac{1}{2} \int_0^t Z^n_s \otimes \dd M^{\Phi_m,
    n}_s - \frac{1}{2} \int_{T - t}^T (Z^n_T - Z^n_{T - s}) \otimes \dd
    \hat{M}^{\Phi_m,n}_s + \frac{1}{2} Z^n_T \otimes (\hat{M}^{\Phi_m,n}_T - \hat{M}^{\Phi_m,n}_{T -
    t})\\
    & \quad + \int_0^t Z^n_s \otimes \sqrt{n} (F (X_{n s}) -\mathcal{L}_S
    \Phi_m (X_{n s})) \dd s.
  \end{align*}
  By Theorem \ref{lem:UCV} together with Lemma~\ref{lem:uniform-tightness-for-path} the two stochastic
  integrals are C-tight in $D ([0, T], \mathbb{R}^d)$ (note that $Z^n_T -
  Z^n_{T - s}$ is adapted to $\hat{\mathcal{F}}_s$). The third term on the
  right hand side is C-tight by the characterization of Lemma
  \ref{lem:C-tight}. It remains to treat the term
\begin{align*}
    & \int_0^t Z^n_s \otimes \sqrt{n} (F (X_{n s}) -\mathcal{L}_S \Phi_m
    (X_{n s})) \dd s\\
    & = \int_0^t \frac{1}{2} (M^{F, n}_s + \hat{M}^{\Phi_m, n}_T - \hat{M}^{F,
    n}_{T - s}) \otimes \sqrt{n} (F (X_{n s}) -\mathcal{L}_S \Phi_m (X_{n s}))
    \dd s\\
    & \quad + \int_0^t \frac{1}{\sqrt{n}} \int_0^{n s} (F (X_r)
    -\mathcal{L}_S \Phi_m (X_r)) \dd r \otimes \sqrt{n} (F (X_{n s})
    -\mathcal{L}_S \Phi_m (X_{n s})) \dd s.
  \end{align*}
  By Corollary \ref{cor:UCV} the integral of the two martingales can
  be handled as before.
  The remaining term satisfies
  \[ \mathbb{E} \left[ \sup_{t \in [0, T]} \left| \frac{1}{n} \int_0^{n t}
     \int_0^s (F (X_r) -\mathcal{L}_S \Phi_m (X_r)) \dd r \otimes (F (X_s)
     -\mathcal{L}_S \Phi_m (X_s)) \dd s \right| \right] \lesssim T \| F
     -\mathcal{L}_S \Phi_m \|_{- 1}^2 < \frac{1}{m^2} \]
  by Lemma~\ref{lem:iterated-kv} in the appendix. Combining all this with the
  necessity of the conditions in Lemma~\ref{lem:C-tight}, we get
\begin{align*}
    & \lim_{K \rightarrow \infty} \limsup_{n \rightarrow \infty} \mathbb{P}
    \left(\sup_{t \in [0, T]} | \mathbb{Z}^n_{0, t} | > K \right)\\
    & \leqslant \lim_{K \rightarrow \infty} \limsup_{n \rightarrow \infty}
    \mathbb{P} \left( \sup_{t \in [0, T]} \left| \frac{1}{n} \int_0^{n t}
    \int_0^s (F (X_r) -\mathcal{L}_S \Phi_m (X_r)) \dd r \otimes (F (X_s)
    -\mathcal{L}_S \Phi_m (X_s)) \dd s \right| > \frac{K}{2} \right) = 0
  \end{align*}
  by Chebyshev's inequality, and similarly we get by bounding $w_T(f,\delta) \le 2 \|f\|_{\infty,[0,T]}$:
\begin{align*}
    & \lim_{\delta \rightarrow 0} \limsup_{n \rightarrow \infty} \mathbb{P} (w_T
    (\mathbb{Z}^n_{0, \cdot}, \delta) > \varepsilon) \\
    & \leqslant \lim_{\delta \rightarrow 0} \limsup_{n \rightarrow \infty}
    \mathbb{P} \left( w_T \left( \frac{1}{n} \int_0^{n \cdot} \int_0^s (F
    (X_r) -\mathcal{L}_S \Phi_m (X_r)) \dd r \otimes (F (X_s)
    -\mathcal{L}_S \Phi_m (X_s)) \dd s, \delta \right) > \frac{\varepsilon}{2}
    \right)\\
    & \lesssim \frac{2}{\varepsilon} T \| F -\mathcal{L}_S \Phi_\lambda \|_{- 1}^2 <
    \frac{2T}{m^2\varepsilon}\to 0 \text{ as }m\to\infty.
  \end{align*}
  Hence
  $\mathbb{Z}^n_{0, \cdot}$ satisfies the assumptions of
  Lemma~\ref{lem:C-tight} and therefore it is tight in $C (\mathbb{R}_+,
  \mathbb{R}^{d \otimes d})$ and the proof is complete.
\end{proof}

To apply Lemma \ref{lem:rp-conv} it remains to show that $\|  (Z^n,
\mathbb{Z}^n) \|_{p, [0, T]}$ is a tight sequence of real valued random
variables. For that purpose we first recall the following estimate:

\begin{lemma}
  \label{lem:kv-p-var}Let $G \in L^2 (\pi) \cap \mathcal{H}^{- 1}$ and $T > 0$
  and $p > 2$. Then
  \[ \mathbb{E} \left[ \sup_{t \leqslant T} \left| \int_0^t G (X_s) \dd s
     \right|^2 + \left\| \int_0^{\cdot} G (X_s) \dd s \right\|_{p, [0,
     T]}^2 \right] \lesssim T \| G \|_{- 1}^2 . \]
\end{lemma}

\begin{proof}
  See Corollary 3.5 in {\cite{Gubinelli2018Energy}}. This corollary is written
  for the specific process studied in {\cite{Gubinelli2018Energy}}, but the
  proof carries over verbatim to the general situation considered here.
\end{proof}

In particular, we get $\mathbb{E} [\| Z^n \|_{p, [0, T]}^2] \lesssim T \| F
\|_{- 1}^2$. To bound $\| \mathbb{Z}^n \|_{p, [0, T]}$ we need the following
auxiliary result, which is the core technical result of this section and which
replaces the Burkholder-Davis-Gundy inequality for local martingale rough
paths of {\cite{Chevyrev2019, friz2018differential}}  in the case where only the integrator
is a local martingale:

\begin{proposition}\label{prop:area-variation}
Let $(Y_t)_{t \in [0, T]}$ be a predictable
  c\`adl\`ag process with $Y_0 = 0$ and such that $\mathbb{E} [\| Y \|_{p,[0, T]}^2] <
  \infty$ for some $p > 2$ and let $(N_t)_{t \in [0, T]}$ be a c\`adl\`ag local
  martingale with $\mathbb{E} [\langle N \rangle_T] < \infty$. Define $A_{s,
  t} \assign \int_s^t Y_{s, r} \dd N_r$. Then for any $q > p > 2$ and for all
  sufficiently small $\varepsilon > 0$
\begin{align*}
    \mathbb{E} [\| A \|_{q / 2, [0, T]}^{1 - \varepsilon}] & \lesssim
    \left(1 +\mathbb{E} \left[\| Y \|_{p, [0, T]}^2\right]^{1 / 2} \right) \mathbb{E} [\langle N \rangle_T]^{1 / 2} .
  \end{align*}
\end{proposition}

To not disrupt the flow of reading we give the proof in Section~\ref{sec:proof-area-variation} below, see in particular the more precise result in Proposition~\ref{prop:area-variation-precise}.

\begin{corollary}
  \label{cor:p-var-iterated-kv}Let $G, H \in \mathcal{H}^{- 1} \cap L^2 (\pi)$
  and set $A_{s, t} = \int_s^t \int_s^{r_1} G (X_{r_2}) \dd r_2 H (X_{r_1})
  \dd r_1$. Then we have for all $p > 2$ and $T > 0$ and $\varepsilon > 0$
  \[ \mathbb{E} [\| A \|_{p / 2, [0, T]}^{1 - \varepsilon}] \lesssim (1 + T^{1
     / 2} \| G \|_{- 1}) (1 + T^{1 / 2} \| H \|_{- 1}) . \]
\end{corollary}

\begin{proof}
  Lemma~\ref{lem:fb-decomp} shows that
\begin{align}\nonumber
    A_{s, t} & = \frac{1}{2} \int_s^t \int_s^{r_1} G (X_{r_2}) \dd r_2
    \dd M^\Psi_{r_1} - \frac{1}{2} \int_{T - s}^{T - t} \int_{T - r_1}^t G
    (X_{r_2}) \dd r_2 \dd \hat{M}^\Psi_{r_1} + \frac{1}{2} \int_s^t G (X_r)
    \dd r (\hat{M}^\Psi_{T - s} - \hat{M}^\Psi_{T - t})\\ \label{eq:p-var-iterated-kv-pr1}
    & \quad + \int_s^t \int_s^{r_1} G (X_{r_2}) \dd r_2 (H (X_{r_1})
    -\mathcal{L}_S \Psi (X_{r_1})) \dd r_1 .
  \end{align}
  The first two terms on the right hand side will be controlled with
  Proposition~\ref{prop:area-variation} and Lemma~\ref{lem:kv-p-var}. The third term of~\eqref{eq:p-var-iterated-kv-pr1} is bounded by
  \begin{equation}
    \label{eq:p-var-iterated-kv-pr2} \left| \frac{1}{2} \int_s^t G (X_r)
    \dd r (\hat{M}^\Psi_{T - s} - \hat{M}^\Psi_{T - t}) \right| \lesssim \left\|
    \int_0^{\cdot} G (X_r) \dd r \right\|_{p, [s, t]} \| \hat{M}^\Psi \|_{p,
    [T - t, T - s]},
  \end{equation}
  and the fourth term by
\begin{equation}\label{eq:p-var-iterated-kv-pr3}
\begin{aligned}
    & \left| \int_s^t \int_s^{r_1} G (X_{r_2}) \dd r_2 (H (X_{r_1})
    -\mathcal{L}_S \Psi (X_{r_1})) \dd r_1 \right|\\
    & \lesssim \sup_{r \in [s, t]} \left| \int_0^r G (X_{r_2}) \dd r_2
    \right| \int_s^t | H (X_{r_1}) -\mathcal{L}_S \Psi (X_{r_1}) | \dd
    r_1 .
\end{aligned}
\end{equation}
	{Recall also L\'epingle's $p$-variation Burkholder-Davis-Gundy inequality, see Theorem~\ref{thm:lepingle BDG}, and note that $\bb E[[M^\Psi]_T] = \bb E[\langle M^\Psi\rangle_T]$ which can be easily seen by stopping the local martingale $[M^\Psi] - \langle M^\Psi\rangle$ and then applying monotone convergence.} So by Proposition~\ref{prop:area-variation} together with \eqref{eq:p-var-iterated-kv-pr1}-\eqref{eq:p-var-iterated-kv-pr3} we obtain
\begin{align*}
    \mathbb{E} [\| A \|_{p / 2, [0, T]}^{1 - \varepsilon}] & \lesssim \left( 1
    +\mathbb{E} \left[ \left\| \int_0^{\cdot} G (X_r) \dd r \right\|_{p,
    [0, T]}^2 \right]^{1 / 2} \right) (1 + |\mathbb{E} [\langle M^\Psi
    \rangle_T]|^{1 / 2} +|\mathbb{E}[\langle \hat{M}^\Psi \rangle_T]|^{1 / 2})\\
    & \quad +\mathbb{E} \left[ \left\| \int_0^{\cdot} G (X_r) \dd r
    \right\|_{p, [0, T]}^2 \right]^{(1 - \varepsilon) / 2} \mathbb{E} [\|
    \hat{M}^\Psi \|_{p, [0, T]}^2]^{(1-\varepsilon) / 2}\\
    & \quad +\mathbb{E} \left[ \sup_{r \in [0, T]} \left| \int_0^r G
    (X_{r_2}) \dd r_2 \right|^2 \right]^{(1 - \varepsilon) / 2} \mathbb{E}
    \left[ \left( \int_0^T | H (X_{r_1}) -\mathcal{L}_S \Psi (X_{r_1}) |
    \dd r_1 \right)^2 \right]^{(1-\varepsilon) / 2}\\
    & \lesssim \left( 1 +\mathbb{E} \left[ \left\| \int_0^{\cdot} G (X_r)
    \dd r \right\|_{p, [0, T]}^2 \right]^{1 / 2} \right) (1 + T^{1 / 2} \|
    \Psi \|_1 + T \| H -\mathcal{L}_S \Psi \|_{L^2 (\pi)})\\
    & \lesssim (1 + T^{1 / 2} \| G \|_{- 1}) (1 + T^{1 / 2} \| \Psi \|_1 +
    T \| H -\mathcal{L}_S \Psi \|_{L^2 (\pi)}),
  \end{align*}
  where the last step follows from Lemma~\ref{lem:kv-p-var}. Now we take
  $\Psi = \Phi^H_{\lambda}$ as the solution to the Poisson equation
  $(\lambda -\mathcal{L}_S) \Phi^H_{\lambda} = - H$. Note that in general
  $\Phi^H_{\lambda} \nin \mathcal{C}$, but we can approximate
  $\Phi^H_{\lambda}$ with functions in $\mathcal{C}$ and get the same
  estimate. Then standard estimates for the solution of the resolvent
  equation, see eq. (2.15) in {\cite{Komorowski2012}}, give $\|
  \Phi^H_{\lambda} \|_1 + \sqrt{\lambda} \| \Phi_{\lambda}^H \|_{L^2 (\pi)}
  \lesssim \| H \|_{- 1}$, and since $H -\mathcal{L}_S \Phi^H_{\lambda} =
  \lambda \Phi^H_{\lambda}$we can send $\lambda \rightarrow 0$ to deduce the
  claimed estimate.
\end{proof}

\begin{corollary}
  The process $(Z^n, \mathbb{Z}^n)$ is tight in the $p$-variation topology on $C (\mathbb{R}_+, \mathbb{R}^d \oplus \mathbb{R}^{d \otimes
  d})$.
\end{corollary}

\begin{proof}
  It remains to show that $(\| \mathbb{Z}^n \|_{p / 2, [0, T]})_n$ is tight
  for all $T > 0$, for which it suffices that $\mathbb{E} [\| \mathbb{Z}^n \|_{p
  / 2, [0, T]}^{1 - \varepsilon}] \leqslant C$ for all $n$. But this follows
  from Corollary~\ref{cor:p-var-iterated-kv}: We set $G = H = n^{- 1 / 2} F$
  and replace $T$ with $n T$ to obtain
  \[ \mathbb{E} [\| \mathbb{Z}^n \|_{p / 2, [0, T]}^{1 - \varepsilon}]
     \lesssim (1 + (n T)^{1 / 2} \| n^{- 1 / 2} F \|_{- 1}) (1 + (n T)^{1 / 2}
     \| n^{- 1 / 2} F \|_{- 1}) = (1 + T^{1 / 2} \| F \|_{- 1})^2 . \]

\end{proof}

\subsection{Identification of the limit}

To prove tightness we worked with the forward-backward decomposition of
Lemma~\ref{lem:fb-decomp}. But since the process $\hat{M}^\Psi$ from that lemma
is only a martingale in the backward filtration, this decomposition is not
useful for identifying the limit. So here we work instead with the following
decomposition based on the resolvent equation:

\begin{lemma}
  \label{lem:resolvent-decomposition}For $\lambda > 0$ we write
  $\Phi_{\lambda}$ for the solution of the resolvent equation $(\lambda
  -\mathcal{L}) \Phi_{\lambda} = F$. Then
  \begin{equation}
    \lambda \| \Phi_{\lambda} \|_{L^2 (\pi)}^2 + \| \Phi_{\lambda} \|_1^2
    \leqslant \| F \|_{- 1}^2
  \end{equation}
  and there exists a martingale $M^{\lambda}$ with $M^{\lambda}_0 = 0$ and
  with $\mathbb{E} [\langle M^{\lambda} \rangle_t] = 2 E_{\pi} [\Phi_{\lambda}
  \otimes (-\mathcal{L}_S) \Phi_{\lambda}] t$, such that
  \[ \int_0^t F (X_s) \dd s = \Phi_{\lambda} (X_0) - \Phi_{\lambda} (X_t) +
     \int_0^t \lambda \Phi_{\lambda} (X_s) \dd s + M^{\lambda}_t
     \backassign R^{\lambda}_t + M^{\lambda}_t . \]
  We write $M^{\lambda, n}_t \assign n^{- 1 / 2} M^{\lambda}_{n t}$ and
  $R^{\lambda, n}_t \assign n^{- 1 / 2} R^{\lambda}_{n t}$.
\end{lemma}

\begin{proof}
  This formally follows by applying Dynkin's formula to $\Phi_{\lambda}$, {and to make it rigorous if $\Phi_\lambda\otimes \Phi_\lambda \nin \operatorname{dom}(\mc L)$ one can use an approximation argument (see p.35 of \cite{Komorowski2012}).}
\end{proof}

\begin{lemma}
  \label{lem:lambda-equal-0}Assume~{\eqref{eq:additive-rp-assumption-1}}. Then
  there exist processes $R^n, M^n \in D (\mathbb{R}_+, \mathbb{R}^d)$ such
  that for all $T > 0$ and $n \in \mathbb{N}$
  \[ \lim_{\lambda \rightarrow 0} \big\{ \mathbb{E} \big[\sup_{t \leqslant T} |
     M^n_t - M_t^{\lambda, n} |^2\big] +\mathbb{E} \big[\sup_{t \leqslant T} | R^n_t -
     R_t^{\lambda, n} |^2\big] \big\} = 0. \]
  Moreover, $M^n$ is a martingale with $\mathbb{E} [\langle M^n \rangle_t] = 2
  t \lim_{\lambda \rightarrow 0} E_{\pi} [\Phi_{\lambda} \otimes
  (-\mathcal{L}_S) \Phi_{\lambda}]$.
\end{lemma}

\begin{proof}
  This is all shown in~{\cite{Komorowski2012}}, see Lemma~2.9 and (2.26)
  therein.
\end{proof}

The following corollary completes the proof of Theorem~\ref{thm:additive-rp}:

\begin{corollary}
  \label{cor:additive-rp}Under the assumptions of
  Theorem~\ref{thm:additive-rp} the process $(Z^n, \mathbb{Z}^n)$ converges in
  distribution in the $p$-variation topology on $C (\mathbb{R}_+,
  \mathbb{R}^d \oplus \mathbb{R}^{d \otimes d})$ to
  \begin{equation}
    \left( B_t, \int_0^t B_s \otimes \circ \dd B_s + \Gamma t \right)_{t \geqslant 0},
  \end{equation}
  where $B$ is a $d$-dimensional Brownian motion with covariance
  \[ 2 t \lim_{\lambda \rightarrow 0} E_{\pi} [\Phi_{\lambda} \otimes
     (-\mathcal{L}_S) \Phi_{\lambda}] = 2 t \lim_{\lambda \rightarrow 0}
     \langle \Phi_{\lambda}, \otimes \Phi_{\lambda} \rangle_1, \]
  and where
  \[ \Gamma = \lim_{\lambda \rightarrow 0} E_{\pi} [\Phi_{\lambda} \otimes
     \mathcal{L}_A \Phi_{\lambda}] . \]
\end{corollary}

\begin{proof}
  Let $Z^n = M^n + R^n$ as above. In Theorem~2.32 of {\cite{Komorowski2012}}
  it is shown that both $(M^n)$ and $ (Z^n)$ converge in distribution in the Skorohod topology on $D
  (\mathbb{R}_+, \mathbb{R}^{d})$ to a Brownian motion $B$ with covariance
  $\langle B,B \rangle_t = 2 t \lim_{\lambda \rightarrow 0} E_{\pi}
  [\Phi_{\lambda} \otimes (-\mathcal{L}_S) \Phi_{\lambda}]$. Therefore both
  $Z^n$ and $M^n$ are C-tight, and thus also $R^n$ is C-tight. It is shown in
  Proposition~2.8 of {\cite{Komorowski2012}} that $\mathbb{E} [| R^n_t |^2]
  \rightarrow 0$ for each fixed $t \geqslant 0$, which together with the
  $C$-tightness gives the convergence of $R^n$ to zero in distribution in $C
  (\mathbb{R}_+, \mathbb{R}^d)$ (and thus in probability because the limit is
  deterministic). Since $Z^n = M^n + R^n$, this gives the joint convergence of $(Z^n, M^n,
  R^n)$ in distribution in $C (\mathbb{R}_+, \mathbb{R}^{3 d})$ to $(B, B,
  0)$. By the `moreover' part of Lemma~\ref{lem:lambda-equal-0} $M^n$ satisfies UCV. Consequently, Corollary~\ref{cor:UCV} shows the joint
  convergence
  \[ \left( Z^n, M^n, \int_0^{\cdot} M^n_s \otimes \dd Z^n_s \right)
     \rightarrow \left( B, B, \int_0^{\cdot} B_s \otimes \dd B_s +
     \langle B, B \rangle \right) . \]
  It remains to study the term $\int_0^{\cdot} R^n_s \otimes \dd Z^n_s$.
  We claim that for all $T > 0$
  \begin{equation}
    \label{eq:identifying-limit-pr1} \lim_{n \rightarrow \infty} \mathbb{E}
    \left[ \sup_{t \leqslant T} \left| \int_0^t (R^n_s + n^{- 1 / 2}
    \Phi_{n^{- 1}} (X_{n s})) \otimes \dd Z^n_s \right| \right] = 0.
  \end{equation}
  Indeed, $R^n_s - R^{n^{- 1}, n}_s = M^{n^{- 1}, n}_s - M^n_s$, and since
  $\mathbb{E} [\sup_{t \leqslant T} | M^n_t - M^{n^{- 1}, n}_t |^2] \lesssim
  \| \Phi - \Phi_{n^{- 1}} \|_1^2 \rightarrow 0$ we can apply integration by
  parts together with the Burkholder-Davis-Gundy inequality to show that
  $\mathbb{E} \left[ \sup_{t \leqslant T} \left| \int_0^t (R^n_s - R^{n^{- 1},
  n}_s) \otimes \dd Z^n_s \right| \right] \rightarrow 0$. The remaining
  term involves only the continuous finite variation process $R^{n^{- 1}, n}_s
  + n^{- 1 / 2} \Phi_{n^{- 1}} (X_{n s})$, so that we can apply
  Lemma~\ref{lem:iterated-kv} to obtain
\begin{align*}
    & \limsup_{n \rightarrow \infty} \mathbb{E} \left[ \sup_{t \leqslant T}
    \left| \int_0^t (R^{n^{- 1}, n}_s + n^{- 1 / 2} \Phi_{n^{- 1}} (X_{n s}))
    \otimes \dd Z^n_s \right| \right]\\
    & = \limsup_{n \rightarrow \infty} \mathbb{E} \left[ \sup_{t \leqslant T}
    \left| \frac{1}{\sqrt{n}} \int_0^{n t} (R^{n^{- 1}, n}_{n^{- 1} s} + n^{-
    1 / 2} \Phi_{n^{- 1}} (X_s)) \otimes F (X_s) \dd s \right| \right]\\
    & \lesssim \limsup_{n \rightarrow \infty} \mathbb{E} [\sup_{t \leqslant n
    T} |  R^{n^{- 1}, n}_{n^{- 1} t} + n^{- 1 / 2} \Phi_{n^{- 1}}
    (X_t)) \dd s |^2]^{1 / 2} T^{1 / 2} \| F \|_{- 1} .
  \end{align*}
  To bound the expectation on the right hand side note that
\begin{align*}
    \mathbb{E} [\sup_{t \leqslant T} | R^{n^{- 1}, n}_t + n^{- 1 / 2}
    \Phi_{n^{- 1}} (X_{n t}) |^2] & \lesssim \mathbb{E} [| n^{- 1 / 2}
    \Phi_{n^{- 1}} (X_0) |^2] +\mathbb{E} \left[ \sup_{t \leqslant T} \left|
    n^{- 1 / 2} \int_0^{n t} n^{- 1} \Phi_{n^{- 1}} (X_s) \dd s \right|^2
    \right]\\
    & \lesssim n^{- 1} \| \Phi_{n^{- 1}} \|_{L^2 (\pi)}^2 + T^2 n^{- 1} \|
    \Phi_{n^{- 1}} \|_{L^2 (\pi)}^2\\
    & = (1 + T^2) n^{- 1} \| \Phi_{n^{- 1}} \|_{L^2 (\pi)}^2,
  \end{align*}
  and since according to assumption~{\eqref{eq:additive-rp-assumption-1}} the
  right hand side vanishes for $n \rightarrow \infty$ we
  deduce~{\eqref{eq:identifying-limit-pr1}}. Therefore, it suffices to study
  the limit of $\int_0^t n^{- 1 / 2} \Phi_{n^{- 1}} (X_{n s}) \otimes \dd
  Z^n_s = n^{- 1} \int_0^{n t} \Phi_{n^{- 1}} (X_{n s}) \otimes F (X_{n s})
  \dd s$. Let $\lambda > 0$, then
\begin{align*}
    \mathbb{E} \left[ \sup_{t \leqslant T} \left| n^{- 1} \int_0^{n t}
    (\Phi_{n^{- 1}} (X_{n s}) - \Phi_{\lambda} (X_{n s})) \otimes F (X_{n s})
    \dd s \right| \right] & \leqslant T E_{\pi} [| (\Phi_{n^{- 1}} -
    \Phi_{\lambda}) \otimes F |]\\
    & \leqslant T \| \Phi_{n^{- 1}} - \Phi_{\lambda} \|_1 \| F \|_{- 1},
  \end{align*}
  and by assumption the right hand side converges to $T \| \Phi -
  \Phi_{\lambda} \|_1 \| F \|_{- 1}$, which goes to zero for $\lambda
  \rightarrow 0$. Moreover, by the ergodic theorem the term $n^{- 1} \int_0^{n
  t} \Phi_{\lambda} (X_{n s}) \otimes F (X_{n s}) \dd s$ converges almost
  surely and in $L^1 (\mathbb{P})$ to $t E_{\pi} [\Phi_{\lambda} \otimes F]$.
  By Lemma~\ref{lem:stationary-uniform} in the appendix this convergence is
  even uniform in $t \in [0, T]$ {(to get the required uniform integrability note that
  \[
  	\sup_{t \in [0,T]} \left|n^{- 1} \int_0^{nt} \Phi_{\lambda} (X_{n s}) \otimes F (X_{n s}) \dd s\right| \le n^{-1}\int_0^{nT} | \Phi_{\lambda} (X_{n s}) \otimes F (X_{n s})| \dd s,
  \]
  and the right hand side converges in $L^1$ by the ergodic theorem).}
  Now it suffices to send $\lambda \rightarrow
  0$ to deduce that $\int_0^t R^n_s \otimes \dd Z^n_s$ converges to the
  deterministic limit $- t \lim_{\lambda \rightarrow 0} E_{\pi}
  [\Phi_{\lambda} \otimes F]$ in $C (\mathbb{R}_+, \mathbb{R}^d)$.
  Consequently,
  \begin{align*}
  	 (Z^n, \mathbb{Z}^n) \rightarrow & \left( B, \int_0^t B_s \otimes \dd B_s + \langle B, B \rangle_t - t \lim_{\lambda \rightarrow 0} E_{\pi}
     [\Phi_{\lambda} \otimes F] \right) \\
     = & \left( B, \int_0^t B_s \otimes \circ \dd B_s + \frac12 \langle B, B \rangle_t - t \lim_{\lambda \rightarrow 0} E_{\pi}
     [\Phi_{\lambda} \otimes F] \right),
  \end{align*}
  and finally we have
  \[
  	\lim_{\lambda \rightarrow 0} E_{\pi} [\Phi_{\lambda}
  \otimes F] = \lim_{\lambda \rightarrow 0} E_{\pi} [\Phi_{\lambda} \otimes
  (\lambda -\mathcal{L}) \Phi_{\lambda}] = \lim_{\lambda \rightarrow 0}
  E_{\pi} [\Phi_{\lambda} \otimes (-\mathcal{L}) \Phi_{\lambda}]
  \]
{because
  $\sqrt{\lambda} \Phi_{\lambda} \rightarrow 0$ in $L^2 (\pi)$. The limit on the left hand side exists because $\Phi_\lambda$ converges in $\mathcal H^1$ and $F\in \mathcal H^{-1}$, and thus also the limit on the right hand side exists. Moreover, $\frac12 \langle B,B \rangle_t = t \lim_{\lambda \rightarrow 0} E_{\pi}
  [\Phi_{\lambda} \otimes (-\mathcal{L}_S) \Phi_{\lambda}]$ and since $\mc L - \mc L_S = \mc L_A$ we get the claimed form $\Gamma = \lim_{\lambda \rightarrow 0} E_{\pi}
  [\Phi_{\lambda} \otimes \mc L_A \Phi_{\lambda}]$ (and in particular this limit exists).}
\end{proof}

\section{Applications}\label{sec:applications}

To illustrate the applicability of our results we derive here scaling limits
in the rough path topology for three classes of models, random walks with random
conductances, Ornstein-Uhlenbeck process with divergence free drift, and diffusions with periodic coefficients.

\subsection{Random walks with random conductances}

We place ourselves in the setting of Chapter~3.1 of {\cite{Komorowski2012}} or
{\cite{Mourrat2012}}. Namely, let
\[ \eta = \{ \eta (\{ x, y \}) = \eta (\{ y, x \}) : x, y \in \mathbb{Z}^d, |
   x - y | = 1 \} \]
be a set of numbers with $0 < c \leqslant \eta (\{ x, y \}) \leqslant C$ for
all $x, y$ and let us write $X^{\eta}$ for the continuous time random walk in
$\mathbb{Z}^d$ with $X^{\eta}_0 = 0$ and that jumps from $x$ to $y$ (resp.
from $y$ to $x$) with rate $\eta (\{ x, y \})$. Since the rates are bounded
from above this random walk exists for all times. We interpret $\eta (\{ x, y
\})$ as the \emph{conductance} on the bond $\{ x, y \}$. To simplify
notation we will write
\[ \eta (x, y) = \eta (y, x) = \eta (\{ x, y \}) \]
from now on. We are interested in the situation where $(\eta (\{ x, y \}))_{| x - y | = 1}$ is an
i.i.d. family of random variables (and each $\eta (x, y)$ still
takes values in $[c, C]$).

\subsubsection{Scaling limit for the It\^{o} rough path}

Let us write $\pi$ for the distribution of $\eta$ and write $X^{\eta}_{t -} =
\lim_{s \uparrow t} X^{\eta}_s$ and then
\[ \mathbb{X}^{\eta}_{s, t} = \int_s^t X^{\eta}_{s, r -} \otimes \dd
   X^{\eta}_r . \]
We also define
\[ X^{\eta, n}_t = n^{- 1 / 2} X^{\eta}_{n t}, \qquad \mathbb{X}^{\eta, n}_{s,
   t} = \int_s^t X^{\eta, n}_{s, r -} \otimes \dd X^{\eta, n}_r . \]
Our aim is to show an invariance principle in the rough path topology for
$(X^{\eta, n}, \mathbb{X}^{\eta, n})$ under the \emph{annealed measure}
\[ \int \mathbb{E} [f (X^{\eta})] \pi (\dd \eta) . \]
The corresponding annealed invariance principle for $X^{\eta}$ in the Skorohod
topology is established in Chapter~3.1 of {\cite{Komorowski2012}}. The
approach there is based on writing $X^{\eta}$ as an additive functional of a
certain Markov process plus a martingale, and on applying
Lemma~\ref{lem:uniform-tightness-for-path} to the additive functional. The
Markov process is the ``environment as seen from the walker'': For $x \in
\mathbb{Z}^d$ let us write
\[ \tau_x \eta (y, z) = \eta (y + x, z + x), \]
and then we define
\[ \eta_t \assign \tau_{X^{\eta}_t} \eta, \]
which is a c\`adl\`ag process with values in the compact space $[c, C]^{E^d}$
equipped with the product topology, where $E^d = \{ \{ x, y \} : x, y \in
\mathbb{Z}^d, | x - y | = 1 \}$ are the bonds in $\mathbb{Z}^d$. We write
\[ \mathcal{F}_t = \sigma (X^{\eta}_s \vee \eta : s \leqslant t), \]
so that $(\eta_t)$ is adapted to $(\mathcal{F}_t)$. In the following all
martingales are with respect to $(\mathcal{F}_t)$ and the annealed measure,
unless explicitly stated otherwise.

\begin{lemma}[{\cite{Komorowski2012}}, Lemma~3.1]
  The process $(\eta_t)_{t \geqslant 0}$ is Markovian with respect to
  $(\mathcal{F}_t)$, with generator
  \[ \mathcal{L}F (\eta) = \sum_{\tmscript{\begin{array}{c}
       y \in \mathbb{Z}^d :\\
       | y | = 1
     \end{array}}} \eta (0, y) (F (\tau_y \eta) - F (\eta)) \]
  and with reversible and ergodic invariant distribution $\pi$.
\end{lemma}

In Lemma~3.1 of {\cite{Komorowski2012}} the filtration with respect to which
the Markov property holds is not specified, but (a slight modification of)
their proof shows that we can take $(\mathcal{F}_t)$ and not just the
canonical filtration of $(\eta_t)$.

Let us define the \emph{local drift} $F : [c, C]^{E^d} \rightarrow
\mathbb{R}^d$ by
\[ F (\eta) = \sum_{| y | = 1} y \eta (0, y) . \]
It is shown on p.86 of {\cite{Komorowski2012}} that there exists a c\`adl\`ag
martingale $(N_t)_{t \geqslant 0}$ such that
\begin{equation}\label{eq:rwre-decomposition}
	 X^{\eta}_t = N_t + \int_0^t F (\eta_s) \dd s = : N_t + Z_t,
\end{equation}
and therefore $X^{\eta, n}_t = N^n_t + Z^n_t$ with the obvious definition of
the rescaled processes $N^n$ and $Z^n$. The idea is now to apply the
invariance principle for additive functionals to $Z^n$ and to apply the
martingale central limit theorem to $N^n$. Recall that $(\eta_t)$ is
reversible, so by the discussion in Chapter~2.7.1 in {\cite{Komorowski2012}}
we have $F \in L^2 (\pi) \cap \mathcal{H}^{- 1}$ and the assumptions of
Theorem~\ref{thm:additive-rp} are satisfied. Of course, we also have to
understand the joint convergence of $(N^n, Z^n)$, and for that purpose on p.88
of {\cite{Komorowski2012}} the predictable quadratic covariation between $N^n$
and the martingale $M^{\lambda, n}$ from the decomposition of
Lemma~\ref{lem:resolvent-decomposition} is derived, namely for $a, b \in
\mathbb{R}$
\begin{equation}
  \label{eq:rwre-qv} \langle a N^n + b M^{n,\lambda}, a N^n + b M^{n,\lambda} \rangle_t = \sum_{| y | = 1}
  \frac{1}{n} \int_0^{n t} \eta_s (0, y) (a y + b (\Phi_{\lambda} (\tau_y
  \eta_s) - \Phi_{\lambda} (\eta_s)))^{\otimes 2} \dd s
\end{equation}
A simple adaptation of Theorem 3.2 in {\cite{Komorowski2012}} now leads to the
following:

\begin{lemma}
  \label{lem:random-conductance-skorohod}Under the annealed measure the pair
  $(N^n, Z^n)$ converges in distribution in the Skorohod topology on $D
  (\mathbb{R}_+, \mathbb{R}^{2 d})$ to a $2 d$-dimensional Brownian motion
  $(B^N, B^Z)$ such that for $a, b \in \mathbb{R}$
  \begin{equation}
    \label{eq:random-conductance-limiting-qv} \langle a B^N + b B^Z, a B^N + b
    B^Z \rangle_t = t \lim_{\lambda \rightarrow 0} \sum_{| y | = 1} E_{\pi}
    [\eta (0, y) (a y + b (\Phi_{\lambda} (\tau_y \eta) - \Phi_{\lambda}
    (\eta)))^{\otimes 2}] .
  \end{equation}
  Moreover, the sequence of processes $(N^n)$ satisfies the UCV condition.
\end{lemma}

Combining this result with Theorem~\ref{thm:additive-rp}, we easily obtain the
following convergence in rough path topology:

\begin{theorem}\label{thm: conductances ito convergence}
  The process $(X^{\eta, n}, \mathbb{X}^{\eta, n})$ converges in distribution
  in the $p$-variation rough path topology to
  \[
  	\left( B, \left( \int_0^t B_s \otimes \dd B_s + \Gamma t \right)_{t \geqslant 0} \right),
  \]
  where $B$ is a Brownian motion with covariance
  \[ \langle B, B \rangle_t = t \lim_{\lambda \rightarrow 0} \sum_{| y | = 1}
     E_{\pi} [\eta (0, y) (y + (\Phi_{\lambda} (\tau_y \eta) - \Phi_{\lambda}
     (\eta)))^{\otimes 2}], \]
  and where for the unit matrix $I_d$ and the vector $e_1 = (1, 0, \ldots, 0)
  \in \mathbb{Z}^d$
  \[ \Gamma = \frac{1}{2} \langle B, B \rangle_1 - E_{\pi} [\eta (0, e_1)] I_d . \]
\end{theorem}

\begin{proof}
  Using~{\eqref{eq:rwre-qv}} together with the arguments from the proof of
  Corollary~\ref{cor:additive-rp} it is not hard to strengthen
  Lemma~\ref{lem:random-conductance-skorohod} to obtain the joint convergence
  \[
  	(N^n, Z^n, \mathbb{Z}^n_{0, \cdot})\longrightarrow \left( B^N, B^Z,\int_0^{\cdot} B^Z_s \otimes \dd B^Z_s + \frac{1}{2} \langle B^Z, B^Z \rangle \right).
  \]
  Since the limit is continuous the triple is even C-tight,
  and therefore by Lemma~\ref{lem:C-tight} also $X^{\eta, n} = N^n + Z^n$
  converges in distribution in the Skorohod topology to $B = B^Z + B^N$, and
  the convergence is jointly with $(N^n, Z^n, \mathbb{Z}^n_{0, \cdot})$. The
  iterated integrals of $X^{\eta, n}$ are given by
  \begin{equation}
    \label{eq:rwre-iterated-integrals} \mathbb{X}^{\eta, n}_{0, t} = \int_0^t
    X^{\eta,n}_{s -} \otimes \dd N^n_s + \int_0^t N^n_{s -} \otimes \dd Z^n_s
    +\mathbb{Z}^n_{0, t} .
  \end{equation}
  Recall from Lemma~\ref{lem:random-conductance-skorohod} that $N^n$ satisfies
  the UCV property. Since $Z^n$ is continuous and of finite variation, we get
  from Theorem \ref{lem:UCV} and Corollary \ref{cor:UCV} the joint convergence
\begin{align*}
    & \left( N^n, Z^n, \mathbb{Z}^n_{0, \cdot}, X^{\eta, n},
    \int_0^{\cdot} X^n_{s -} \otimes \dd N^n_s, \int_0^{\cdot} N^n_{s
    -} \otimes \dd Z^n_s \right)\\
    & \rightarrow \left( B^N, B^Z, \int_0^{\cdot} B^Z_s \otimes \dd
    B^Z_s + \frac{1}{2} \langle B^Z, B^Z \rangle, B, \int_0^{\cdot} B_s
    \otimes \dd B^N_s, \int_0^{\cdot} B_s^N \otimes \dd B^Z_s +
    \langle B^N, B^Z \rangle \right) .
  \end{align*}
  Since all the limiting processes are continuous the tuple is C-tight and the
  joint convergence extends to sums of the entries, so from
  {\eqref{eq:rwre-iterated-integrals}} we get
\begin{align*}
    (X^{\eta, n}, \mathbb{X}^{\eta, n}_{0, \cdot}) & \rightarrow \left( B,
    \int_0^{\cdot} B_s \otimes \dd B_s + \frac{1}{2} \langle B^Z, B^Z
    \rangle + \langle B^N, B^Z \rangle \right)\\
    & = \left( B, \int_0^{\cdot} B_s \otimes \dd B_s + \frac{1}{2}
    \langle B, B \rangle - \frac{1}{2} \langle B^N, B^N \rangle \right)
  \end{align*}
  and by~{\eqref{eq:random-conductance-limiting-qv}} the last term on
  the right hand side is given by
  \[ - \frac{1}{2} \langle B^N, B^N \rangle_t = - \frac{t}{2} \sum_{| y | = 1}
     E_{\pi} [\eta (0, y) y \otimes y] = - \frac{t}{2} E_{\pi} [\eta (0, e_1)]
     \sum_{| y | = 1} y \otimes y = - t E_{\pi} [\eta (0, e_1)] I_d . \]

  To complete the proof it remains to show tightness of the
  $p$-variation. Since \eqref{eq:additive-rp-assumption-1} holds in the reversible case, see \cite[Section 2.7.1]{Komorowski2012}, Theorem~\ref{thm:additive-rp} implies that $(\|
  Z^n \|_{p, [0, T]} + \| \mathbb{Z}^n \|_{p / 2, [0, T]})_n$ is tight. For
  the first level of the rough path we have $\| X^{\eta, n} \|_{p, [0, T]} \leqslant \| N_n
  \|_{p, [0, T]} + \| Z^n \|_{p, [0, T]}$, and $\mathbb{E} [\| N^n \|_{p, [0,
  T]}^2] \lesssim \mathbb{E} [\langle N^n \rangle_T] \lesssim 1$ by Theorem \ref{thm:lepingle BDG} together with~{\eqref{eq:rwre-qv}},
  and we already know that $\| Z^n \|_{p, [0, T]}$ is tight. From~{\eqref{eq:rwre-iterated-integrals}} we get
\begin{align*}
    \| \mathbb{X}^{\eta, n} \|_{p / 2, [0, T]} & \leqslant \left\| \left(
    \int_s^t X^{\eta,n}_{r, s -} \otimes \dd N^n_r \right)_{0 \leqslant s
    \leqslant t \leqslant T} \right\|_{p / 2, [0, T]} + \left\| \left(
    \int_s^t N^n_{r, s} \otimes \dd Z^n_r \right)_{0 \leqslant s \leqslant
    t \leqslant T} \right\|_{p / 2, [0, T]}\\
    & \quad + \| \mathbb{Z}^n \|_{p / 2, [0, T]} .
  \end{align*}
  We apply Proposition~\ref{prop:area-variation} for the first term on the
  right hand side and obtain
  \[ \mathbb{E} \left[ \left\| \left( \int_s^t X^{\eta,n}_{r, s -} \otimes \dd
     N^n_r \right)_{0 \leqslant s \leqslant t \leqslant T} \right\|_{p / 2,
     [0, T]}^{1 - \varepsilon} \right] \lesssim (1 +\mathbb{E} [\| X^{\eta,n} \|_{p',
     [0, T]}^2]^{1 / 2}) (1 +\mathbb{E} [\langle N^n \rangle_T]^{1 / 2})
     \lesssim 1, \]
  where $p' \in (2, p)$. The second term on the right hand side can be
  controlled via integration by parts and a similar application of
  Proposition~\ref{prop:area-variation}. And we already know that $(\|
  \mathbb{Z}^n \|_{p / 2, [0, T]})_n$ is tight. Hence we get the tightness of
  $(\| \mathbb{X}^{\eta, n} \|_{p / 2, [0, T]})_n$, and this concludes the
  proof.
\end{proof}

{
\begin{remark}
	We did not really use that the conductances are i.i.d., and the same proof works if they are only ergodic with respect to the shifts on $\bb Z^d$. In that case the correction $\Gamma$ of Theorem~\ref{thm: conductances ito convergence} is given by
	\[
		\Gamma = \frac12 \langle B, B \rangle_1 - \operatorname{diag}(E_\pi[\eta(0,e_1)], \dots, E_\pi[\eta(0,e_d)]),
	\]
	where $\operatorname{diag}(\dots)$ is a diagonal matrix with the respective entries on the diagonal.
	In the i.i.d. setting and for $d > 2$ we expect that it is possible to get stronger results (H\"older topology instead of $p$-variation, speed of convergence, convergence under the quenched measure) by using the spectral gap result of~\cite{Gloria2015}.
\end{remark}
}

\subsubsection{Scaling limit for the Stratonovich rough path}
In our discrete setting of the random walk in random environment it seems natural to consider the It\^o iterated integrals $\int_s^t X^{\eta,n}_{s,r-} \otimes \dd X^{\eta,n}_r$. But of course this is not the only option, and we might also turn $X^\eta$ into a continuous path by connecting the jumps piecewise linearly, as it is often done for Donsker's invariance principle. More precisely, if $\sigma_k^n$, $k=1,2,\dots$ are the
jump times of the process $X^{\eta, n}$, then we set
\[
 \bar X^{\eta, n}_t :=  X^{\eta, n}_{\sigma^n_k} + \frac{t-\sigma^n_{k}}{\sigma^n_{k+1} - \sigma^n_{k}}X^{\eta, n}_{\sigma^n_k,\sigma^n_{k+1}}
\]
for $\sigma^n_k \leqslant t < \sigma^n_{k+1}$. We then define $\bar{\mathbb{X}}^{\eta,n}_{s,t}=\int_s^t \bar X^{\eta,n}_{s,r}\otimes \dd \bar X^{\eta,n}_{r}$, and as usual $\bar{\mathbb{X}}^{\eta,n}_{t}=\bar{\mathbb{X}}^{\eta,n}_{0,t}$. Note that $\sup_{t \geqslant 0} | \bar X^{\eta,n}_t - X^{\eta,n}_t | \leqslant n^{-1/2}$, and therefore $\bar X^{\eta,n}$ converges to the same Brownian motion as $X^{\eta,n}$. The difference arises only on the level of the iterated integrals: We have
  \begin{eqnarray*}
   \int_{\sigma^n_k}^{\sigma^n_{k+1}} \bar X^{\eta,n}_{0,r}\otimes \dd \bar X^{\eta,n}_{r} &=&
  \frac 1 2 (\sigma^n_{k+1}-\sigma^n_k) X^{\eta,n}_{\sigma^n_k,\sigma^n_{k+1}} \otimes \frac{X^{\eta,n}_{\sigma^n_k,\sigma^n_{k+1}}}{\sigma^n_{k+1}-\sigma^n_k} +
  X^{\eta,n}_{\sigma^n_{k}}  \otimes X^{\eta,n}_{\sigma^n_{k},\sigma^n_{k+1}}\\
  &=& \frac 1 2 (X^{\eta,n}_{\sigma^n_{k}}+X^{\eta,n}_{\sigma^n_{k+1}})  \otimes X^{\eta,n}_{\sigma^n_{k},\sigma^n_{k+1}},
  \end{eqnarray*}
Therefore,
    \begin{eqnarray*}
    \bar{\mathbb{X}}^{\eta,n}_{\sigma^n_{k}} - {\mathbb{X}}^{\eta,n}_{\sigma^n_{k}} &=&
    \sum_{j=0}^{k-1}  \left( \int_{\sigma^n_j}^{\sigma^n_{j+1}} \bar X^{\eta,n}_{0,r}\otimes \dd \bar X^{\eta,n}_{r} -  X^{\eta,n}_{\sigma^n_{j}}  \otimes X^{\eta,n}_{\sigma^n_{j},\sigma^n_{j+1}} \right)\\
    &=& \frac 1 2 \sum_{j=0}^{k-1} (X^{\eta,n}_{\sigma^n_{j},\sigma^n_{j+1}})^{\otimes 2}.
    \end{eqnarray*}
Using the ergodic theorem for the stationary ergodic sequence
$\left(\{\eta(s):s\in [k,k+1]\}\right)_{k\geqslant 0}$ and the fact that $\sigma^n_k = n\sigma^1_k$, we get for fixed $t$
\begin{eqnarray*}
\frac 1 2 \sum_{k: \sigma^n_k \leqslant       t} (X^{\eta,n}_{\sigma^n_k,\sigma_{k+1}^n})^{\otimes 2} &=&
\frac{1}{2n}  \sum_{ 0< s \leqslant       nt} (X^\eta_{s-,s})^{\otimes 2}\\
&=&
\frac{\lfloor nt\rfloor}{2n} \frac{1}{\lfloor nt\rfloor} \sum_{ k=0}^{\lfloor nt\rfloor-1}
\left(
\sum_{k\leqslant       s < k +1}
\left(\sum_{|y|=1} y \mathrm{1}_{\eta(s)=\tau_y\eta(s-)}  \right)^{\otimes 2}
\right) \\
&& + \frac{1}{2n}\sum_{\lfloor nt \rfloor \leqslant       s < nt} (X^\eta_{s-,s})^{\otimes 2}\\
&=&
\frac{\lfloor nt\rfloor}{2n} \frac{1}{\lfloor nt\rfloor} \sum_{ k=0}^{\lfloor nt\rfloor-1} \Psi(\theta_k\{\eta(s):s\in [0,1]\}) \; + \; O\left(\frac1n\right)\\
&\rightarrow&
\frac{t}{2} \bb E_\pi[\Psi(\{\eta(s):s\in [0,1]\})]\\
&=&
\frac{t}{2} \bb E_{\pi}
\left[\sum_{0<j:\sigma_j^1\leqslant       1}
(X^\eta_{\sigma_{j-1}^1,\sigma_{j}^1})^{\otimes 2}
\right]
\end{eqnarray*}
where the convergence as $n\to\infty$ is {in $L^1(\bb P_\pi)$ (easy to see) and $\bb P_{\pi}$ almost surely (to be justified below),}  and where
\begin{equation}
\Psi(\{\eta(s):s\in [0,1]\}) = \sum_{0<   s \leqslant 1}
\left(\sum_{|y|=1} y \mathds{1}_{\eta(s)=\tau_y\eta(s-)}  \right)^{\otimes 2}.
\end{equation}
{To see that the $O(\frac1n)$ term converges $\bb P_\pi$ almost surely to zero, note that by stationarity and since the norm of $t\mapsto \sum_{0 < s \leqslant t}
\left(\sum_{|y|=1} y \mathds{1}_{\eta(s)=\tau_y\eta(s-)}  \right)^{\otimes 2}\in \bbR^{d\times d}$ is increasing in $t$:
\[
	\bb E_\pi\left[ \left|\frac{1}{2n}\sum_{\lfloor nt \rfloor \leqslant s < nt} (X^\eta_{s-,s})^{\otimes 2}\right|^2\right] \le \frac{1}{4n^2} \bb E_\pi\left[ \left|\sum_{0 \leqslant s < 1} (X^\eta_{s-,s})^{\otimes 2}\right|^2\right].
\]
Our jump rates are uniformly bounded and the size of each jump is bounded by $1$, and therefore the expectation on the right hand side is finite. Consequently,
\[
	\bb E_\pi \left[\sum_n \left|\frac{1}{2n}\sum_{\lfloor nt \rfloor \leqslant s < nt} (X^\eta_{s-,s})^{\otimes 2}\right|^2 \right] < \infty,
\]
and thus the summands converge almost surely to zero.}

By Lemma~\ref{lem:stationary-uniform} in the appendix the $L^1(\pi)$-convergence holds even locally uniformly in time. Let us compute the limit: Since the additive functional in the decomposition of $X^\eta$ in \eqref{eq:rwre-decomposition} does not jump we have $X^\eta_{t-,t} = N_{t-,t}$ for all $t>0$, and therefore
 \begin{eqnarray*}
\bb E_{\pi} \left[
\sum_{0<j:\sigma_j^1\leqslant 1}
(X^\eta_{\sigma_{j-1}^1,\sigma_{j}^1})^{\otimes 2}
\right] &=& \bb E_{\pi}\left[ \sum_{0<t \leqslant       1} (X^\eta_{t-,t})^{\otimes 2} \right] =
E\left[ \sum_{0<t \leqslant  1} (N_{t-,t})^{\otimes 2} \right] \\
&=& \bb E_{\pi}[ [N,N]_1 ] = \bb E_{\pi}[ \langle N,N\rangle_1 ] =  \langle B^N,B^N\rangle_1.
 \end{eqnarray*}
Therefore,
    \begin{equation}\label{eq: qv stratonovich}
    \left(\bar{\mathbb{X}}^{\eta,n}_{t} - {\mathbb{X}}^{\eta,n}_{t}\right)_{t\in[0,T]} \to \frac{1}{2}\big(\langle B^N, B^N\rangle_t \big)_{t\in[0,T]}
    \text{ as } n\to\infty,
    \end{equation}
{and since the left hand side is increasing {in the sense of positive definite matrices and thus in norm} and the convergence is uniform in $L^1$, it holds even for the 1-variation norm.
In the proof of Theorem~\ref{thm: conductances ito convergence} we saw that
\[
	\frac12 \langle B^N, B^N\rangle_t = \frac12 \langle B, B\rangle_t - \Gamma t,
\]
}so together with \eqref{eq: qv stratonovich} and the fact that the Stratonovich integral equals
 $\int_0^t B_s \otimes \circ \dd B_s = \int_0^t B_s \otimes \dd B_s + \frac{1}{2} \langle B,B \rangle_t$,
we deduce for the case of linear interpolations that the limit is the Stratonovich Brownian rough path, with no correction:
\begin{corollary}
  Let $(\bar {X}^{\eta, n}, \bar{\mathbb{X}}^{\eta, n})$ be the linear interpolation of the path $X^{\eta, n}$ and its corresponding iterated integral defined above.
  Then $(\bar {X}^{\eta, n}, \bar{\mathbb{X}}^{\eta, n})$ converges in distribution
  in the $p$-variation topology to $\left( B,  \int_0^\cdot B_s \otimes \circ \dd B_s  \right)$, where $B$ is the same Brownian motion as in Theorem~\ref{thm: conductances ito convergence}, and
 $\circ$ denotes Stratonovich integration.
\end{corollary}

\subsection{Additive functional of Ornstein-Uhlenbeck process with divergence-free drift}

{In this section we give a simple example of an additive functional of an Ornstein-Uhlenbeck process with divergence free drift with a non-vanishing area anomaly, i.e. so that the correction $\Gamma$ from \eqref{eq:limit of af} is non-zero. Let $U:\bbR^2\to\bbR$ be given by $U(x) = \frac12 |x|^2 - \log 2\pi$ and let
\[
        b(x) \;=\;
        \begin{pmatrix}
        -x_2  \\
        x_1
        \end{pmatrix}
        e^{-U(x)}
        \;=\; A x e^{-U(x)},\qquad \text{where } A=\begin{pmatrix}
                 0 & -1 \\
                 1 & 0
            \end{pmatrix}.
\]
Note that $b$ is divergence free. We define the operator
\[
\cL f \;=\; \Delta f - \nabla U \cdot \nabla f - b e^U\cdot \nabla f
\;=\; e^U \nabla \cdot ( e^{-U} \nabla f) - b e^U\cdot \nabla f,
\]
which is the generator of the Ornstein-Uhlenbeck process
\[
	d X_t = -\nabla U(X_t)\dd t - b(X_t) e^{U(X_t)} \dd t + \sqrt 2 \dd W_t = -(I + A) X_t \dd t + \sqrt 2 \dd W_t.
\]
One can check that $\pi(\dd x)=e^{-U(x)}\dd x$ is invariant for $X$. Indeed, if $f\in C^2_b(\bbR^d)$, then integration by parts yields
\[
\int \cL f e^{-U} \dd x \;=\; \int \left(\nabla \cdot (e^{-U}\nabla f) - b\cdot \nabla f\right) \dd x \;=\;
\int \left(\nabla 1 \cdot (e^{-U}\nabla f) + f\nabla \cdot b\right) \dd x \;=\; 0,
\]
because $\nabla \cdot b = 0$. We consider $X$ started in the invariant measure and we are interested in the rough path limit of
\[
	Z^n_t = \frac{1}{\sqrt n}\int_0^{nt} X_s ds.
\]
For that purpose let $F(x) = x$. Since $X$ has a spectral gap, it converges exponentially fast to its invariant measure and we can directly the Poisson equation $-\cL \Phi = F$, i.e. there is no need to first consider the resolvent equation $(\lambda-\cL) \Phi_\lambda = F$ and then send $\lambda \to 0$. To compute the explicit solution to the Poisson equation, we use the standard ansatz
    \[
    \Phi(x) \;=\; C x \;=\;
                \begin{pmatrix}
                 C_{11} & C_{12} \\
                 C_{21} & C_{22}
               \end{pmatrix}
               \begin{pmatrix}
                 x_1  \\
                 x_2
               \end{pmatrix}
               \;=\;
               \begin{pmatrix}
                 C_{11}x_1 + C_{12}x_2   \\
                 C_{21}x_1 + C_{22}x_2
               \end{pmatrix}.
               \]
Write $\Phi(x)= \begin{pmatrix}
                          \Phi_1(x) \\
                          \Phi_2 (x)
                        \end{pmatrix}$, then
$\nabla\Phi(x)= \begin{pmatrix}
                          \nabla\Phi_1(x) \\
                          \nabla\Phi_2 (x)
                        \end{pmatrix}
                        =
                        \begin{pmatrix}
                          C_{11} & C_{12} \\
                          C_{21} & C_{22}
                        \end{pmatrix}
                        = C$
for $\Phi(x)=Cx$.
Hence, for $j=1,2$,
\[
	-\mc L \Phi_j(x) =  \left((I + A)x - \nabla \right) \cdot \nabla \Phi_j(x) = (I + A)x \cdot C_{j,{\cdot}},
\]
or more compactly
\[
	-\mc L \Phi(x) = C(I + A)x.
\]
The equation $-\mc L \Phi(x) = F(x)=x$ then yields
\[
	 C(I + A) = I.
\]
Since $A^2 = -I$, this implies
\[
C= \frac12 (I-A)= \frac12 \begin{pmatrix}
                 1 & 1 \\
                 -1 & 1
            \end{pmatrix}.
\]
The Ornstein-Uhlenbeck operator has a spectral gap, so condition~\eqref{eq:additive-rp-assumption-1} is satisfied; see \cite[Theorem~2.18]{Komorowski2012}. By Theorem~\ref{thm:additive-rp}, we get the convergence
\[
	(Z^n, \bb Z^n) \longrightarrow \left(B, \left( \int_0^t B_s \circ dB_s + \Gamma t\right)_{t\in [0,T]}\right),
\]
in distribution in $p$-variation, where $B$ is a Brownian motion with covariance
\[
	\langle B^i, B^j\rangle_t = 2 t E_\pi[\Phi^i  (-\mc L_S) \Phi^j] = 2t E_\pi [ (Cx)^i (Cx)^j] = 2t (C_{i1}C_{j1} + C_{i2} C_{j2}) = t I,
\]
where we used that under $\pi$ the coordinates $(x^1,x^2)$ are independent standard Gaussians, and where
\begin{align*}
	\Gamma_{ij} & = E_\pi[\Phi^i \mc L_A \Phi^j] = E_\pi[ (Cx)^i (-CAx)^j] = \begin{pmatrix} C_{i1} & C_{i2} \end{pmatrix} E_\pi[ x (Ax)^T] \begin{pmatrix} C_{j1} \\ C_{j2} \end{pmatrix} \\
	& = \begin{pmatrix} C_{i1} & C_{i2} \end{pmatrix} E_\pi\left[ \begin{pmatrix}
                 -x_1x_2 & x_1^2 \\
                 -x_2^2 & x_1x_2
            \end{pmatrix} \right] \begin{pmatrix} C_{j1} \\ C_{j2} \end{pmatrix}
      = -C_{i2} C_{j1} + C_{i1} C_{j2}
      = {\begin{pmatrix} 0 & -1\\ 1 & 0 \end{pmatrix}_{ij}}.
\end{align*}
In other words, we see a nontrivial correction to the iterated integrals of $B$.}

\subsection{Diffusions with periodic coefficients}

Consider a smooth $\mathbb{Z}^d$-periodic function $a : \mathbb{R}^d
\rightarrow \mathbb{R}^{d \times d}$ and
$ L = \nabla \cdummy (a \nabla)$,
that is
\[ Lf(x)
    = \sum_{i, j = 1}^d \left(a_{i j}(x) \partial_i
   \partial_j f(x) + \partial_i a_{i j}(x) \partial_jf(x)\right).
   \]
We assume that the symmetric part of $a$ is uniformly elliptic ($a$ itself is not necessarily symmetric). Then there is a unique diffusion process
associated with $L$, with coefficients
\[ d X_t^j = \sum_{i = 1}^d \partial_i a_{i j} (X_t) d t + \sqrt{2} \sum_{i =
   1}^d \sigma_{j i} (X_t) d W^i_t, \]
where
\[
	\sigma = \sqrt{a^S}, \qquad a^S = \frac{1}{2} (a + a^{\ast}),\qquad a^A = \frac12 (a - a^\ast).
\]
To simplify notation we write
\[ b_j = \sum_{i = 1}^d \partial_i a_{i j} = \nabla \cdummy a_{\cdot j} \]
so
\[ d X_t = b (X_t) d t + \sqrt{2} \sigma (X_t) d W_t . \]
We assume that $X_0$ is uniformly distributed on $[-\frac12,\frac12]^d$ (just so that the Markov process $Y$ below is stationary) and we want to understand the large scale behavior of $X$
in rough path topology, for which we will derive the following result:

\begin{theorem}\label{thm:periodic-diffusion}
  Let
  \[
  	X^n_t = n^{- 1/2} X_{n t},\qquad t \in [0,T].
  \]
  Then the following convergence holds in $p$-variation rough path topology:
  \[ \left( X^n_t, \int_0^t X^n_s \otimes \circ \dd X^n_s \right) \rightarrow
     \bigg( B_t, \int_0^t B_s \otimes \circ \dd B_s + t \underbrace{\int (\nabla \Phi^i  \cdummy  (-a^A)
     \nabla \Phi^j)_{ij}  \dd x}_{=:\Gamma} \bigg), \]
  where $\Phi$ solves the Poisson equation
  \[ - \nabla \cdummy (a \nabla \Phi) = b \]
  and $B$ is a Brownian motion with quadratic variation
  \[ \langle B^i, B^j\rangle_t = 2 t \int (\nabla \Phi^i + e_i) \cdot a^S (\nabla \Phi^j + e_j) \dd x,
  \]
  for the standard basis $(e_1,\dots, e_d)$ of $\bb R^d$. For the It{\^o} rough path we see an additional correction:
  \[ \left( X^n_t, \int_0^t X^n_s \otimes \dd X^n_s \right) \rightarrow \left(
     B_t, \int_0^t B_s \otimes \dd B_s + \frac{1}{2} \langle B, B \rangle_t -
     t\int a^S \dd x - t \int (\nabla \Phi^i  \cdummy  a^A
     \nabla \Phi^j)_{ij}  \dd x \right) . \]
\end{theorem}

\begin{remark}
The convergence of the Stratonovich rough path was previously shown by Lejay and Lyons \cite[Proposition 6]{lejay2003importance}. Their proof uses the fact that we control all moments of $X$, from where the required tightness in H\"older topology (which is stronger than $p$-variation) easily follows via a Kolmogorov continuity criterion for rough paths, and there is no need to invoke a result like Proposition~\ref{prop:area-variation}. Our general approach has the advantage that it applies to a much wider class of models and that we can apply it without having to do additional estimations, but in this special case it gives a weaker result.
\end{remark}


We now sketch the proof of the claimed convergence. Let us rewrite
\[ Y_t = X_t \tmop{mod} \mathbb{Z}^d . \]
Since the coefficients of $X$ are periodic, $Y$ is a Markov process with
values in $\mathbb{T}^d = (\bb R / \bb Z)^d$, with generator $\mathcal{L}$ given by the same
expression as $L$, except now it acts on $C^2 (\mathbb{T}^d)$ rather than on
$C^2 (\mathbb{R}^d)$. The Lebesgue measure on $\mathbb{T}^d$ is invariant for $Y$,
and we have
\[ \int_{\mathbb{T}^d} b_j (x) \dd x = \int_{\mathbb{T}^d} \sum_{i = 1}^d
   \partial_i a_{i j} (x) \dd x = 0 \]
by the periodic boundary conditions.

Therefore, we can write (slightly abusing notation by also considering $b$,
$\sigma$ etc. as functions on $\mathbb{T}^d$)
\[ n^{-1/2} X_{n t} = n^{-1/2} \int_0^{n t} b (Y_s) d s +
   n^{-1/2} \int_0^{n t} \sqrt{2} \sigma (Y_s) d W_s = Z^n_t + M^n_t, \]
where $M^n$ is a martingale with quadratic variation
\[ \langle M^n \rangle_t^{i j} = \frac{2}{n} \int_0^{n t} (\sigma
   \sigma^{\ast})_{i j} (X_s) d s = \frac{2}{n} \int_0^{n t} a^S_{i j}
   (X_s) d s = \frac{2}{n} \int_0^{n t} a^S_{i j} (Y_s) d s, \]
and where $Z^n$ is a functional that we can control with our tools from Section~\ref{sec:kipnis-varadhan}. By the uniform ellipticity of $a^S$ together with the
 Poincar{\'e} inequality we have for all $f$ with $\int f \dd x = 0$
\[ \int f (-\mathcal{L}) f \dd x = \int \nabla f \cdummy a \nabla f \dd x = \int
   \nabla f \cdummy a^S \nabla f \dd x \gtrsim \int | \nabla f |^2 \dd x
   \gtrsim \int f^2 \dd x, \]
i.e. $\mathcal{L}$ has a spectral gap and $Y$ is exponentially ergodic. Thus $\langle M^n \rangle_t \rightarrow 2 t \int a^S (x) \dd x$, from where we can show with some more work that $M^n \rightarrow B^M$ for a $d$-dimensional Brownian motion with covariance
\[ \langle B^M, B^M \rangle_t = 2 t \int a^S (x) d x. \]
Since $M^n$ satisfies the UCV condition, the convergence of the lifted path also holds in
the $p$-variation rough path topology for every $p>2$ by Proposition~\ref{prop:area-variation} or, since both integrator and integrand are martingales, also by \cite[Theorem 6.1]{friz2018differential}. To control the term $Z^n$ we use that $Y$ has a spectral gap and that therefore we can directly solve the resolvent equation with $\lambda=0$, i.e. we consider the solution $\Phi$ to the Poisson equation
\[ -\mathcal{L} \Phi = - \nabla \cdummy (a \nabla \Phi) = b, \]
which is given by $\Phi = \int_0^\infty P_t b \dd t$, where $(P_t)$ is the semigroup of $Y$. The time integral converges because $P_t b$ converges exponentially fast to $\int b \dd x = 0$. Since $Y$ has a spectral gap, the conditions of Theorem~\ref{thm:additive-rp} are satisfied (see \cite[Theorem~2.18]{Komorowski2012}), and therefore $(Z^n, \int_0^\cdot Z^n_s \otimes \dd Z^n_s)$ converges to the corrected Stratonovich rough path
\[ \left( B^Z, \int_0^{\cdummy} B^Z_s \otimes \circ \dd B^Z_s + t \int \Phi
   \otimes \mathcal{L}_A \Phi \dd x \right), \]
where $B^Z$ is a Brownian motion with covariance
\[ \langle B^Z, B^Z \rangle_t^{ij} = 2 \int \Phi^i (-\mathcal{L}_S) \Phi^j
    \dd x = 2 \int \nabla \Phi^i \cdummy  a^S  \nabla \Phi^j  \dd
   x, \]
and where
\[
	\left(\int \Phi \otimes \mathcal{L}_A \Phi \dd x\right)_{ij} = - \int \nabla \Phi^i \cdot (a^A \nabla \Phi^j) \dd x
\]
It remains to understand the quadratic covariation of $B^M$ and $B^Z$, as well
as the cross-integrals $\int Z^n \otimes d M^n$ and $\int M^n \otimes d Z^n$. To derive the covariation, note that we get with the solution to the Poisson
equation $\Phi$
\[ Z^n_t = \frac{1}{n} \Phi (Y_0) - n^{-1/2} \Phi (Y_{n t}) + n^{-1/2}
   \int_0^{n t} \sqrt{2} \sum_{j, i} \partial_j \Phi (X_s) \sigma_{j i}
   (X_s) \dd W^i_s = R^n_t + N^n_t, \]
and the covariation of $N^n$ and $M^n$ is thus given by
\begin{align*}
  \langle N^n, M^n \rangle^{i j}_t & = \frac{2}{n} \int_0^{n t} \sum_{k,
  \ell} \partial_k \Phi^i (X_s) \sigma_{k \ell} (X_s) \sigma_{j \ell} (X_s) \dd
  s\\
  & = \frac{2}{n} \int_0^{n t} \sum_k \partial_k \Phi^i a^S_{k j} (X_s) \dd
  s\\
  & \rightarrow 2 t \int \sum_k \partial_k \Phi^i a^S_{k j}  \dd x,
\end{align*}
so that $B = B^Z + B^M$ is a Brownian motion with covariance
\begin{align*}
  \langle B, B \rangle_t & = \langle B^M, B^M \rangle_t + \langle B^Z, B^Z
  \rangle_t + 2 \langle B^Z, B^M \rangle_t\\
  & = 2 t \int \left( a^S_{ij} + \nabla \Phi^i \cdummy a^S
  \nabla \Phi^j+ 2\sum_k \partial_k \Phi^i  a^S_{k j} \right)_{ij} \dd x \\
  & = 2 t \int \left( (\nabla \Phi^i + e_i) \cdot a^S (\nabla \Phi^j + e_j) \right)_{ij} \dd x
\end{align*}
The cross-iterated integrals satisfy according to Theorem~\ref{lem:UCV} and Corollary~\ref{cor:UCV}
\begin{align*}
  \int_0^{\cdummy} Z^n_s \otimes d M^n_s & \rightarrow \int_0^{\cdummy} B^Z_s
  \otimes d B^M_s,\\
  \int_0^{\cdummy} M^n_s \otimes d Z^n_s & \rightarrow \int_0^{\cdummy} B^M_s \otimes d B^Z_s + \langle B^M, B^Z \rangle - 0,
\end{align*}
so that overall
\begin{align*}
	& \left( X^n_t, \int_0^t X^n_s \otimes \dd X^n_s \right)\\
	&\hspace{30pt} \rightarrow  \left(B_t, \int_0^t B_s \otimes \dd B_s + \frac{1}{2} \langle B^Z, B^Z \rangle_t +
   \langle B^M, B^Z \rangle_t + t \int \Phi (x) \otimes \mathcal{L}_A \Phi (x)
   \dd x \right),
\end{align*}
and the first part of the correction can be further simplified to
\[ \frac{1}{2} \langle B^Z, B^Z \rangle_t + \langle B^M, B^Z \rangle_t =
   \frac{1}{2} \langle B, B^Z \rangle_t + \frac{1}{2} \langle B^M, B^Z
   \rangle_t = \frac{1}{2} \langle B, B \rangle_t - \frac{1}{2} \langle B^M,
   B^M \rangle_t, \]
which finally yields the limit
\[
	\int_0^t X^n_s \otimes \dd X^n_s \to \int_0^t B_s \otimes \dd B_s + \frac{1}{2} \langle B, B \rangle_t -
     t \int a^S \dd x - t \int (\nabla \Phi^i  \cdummy  a^A
     \nabla \Phi^j)_{ij}  \dd x.
\]
Tightness in $p$-variation follows as in the example of the random conductance model. This proves the first claim of Theorem~\ref{thm:periodic-diffusion}, about the limit of the It\^o rough path.

To identify the limit of the Stratonovich rough path we use that $\langle X^n, X^n \rangle =
\langle M^n, M^n \rangle$ and thus
\begin{align*}
  \int_0^t X^n_s \otimes \circ \dd X^n_s & = \int_0^t X^n_s \otimes \dd X^n_s +
  \frac{1}{2} \langle X^n, X^n \rangle_t\\
  & = \int_0^t X^n_s \otimes \dd X^n_s + \frac{1}{2} \langle M^n, M^n
  \rangle_t\\
  & \rightarrow \int_0^t B_s \otimes \dd B_s + \frac{1}{2} \langle B, B
  \rangle_t - \frac{1}{2} \langle B^M, B^M \rangle_t + t \int \Phi (x) \otimes
  \mathcal{L}_A \Phi (x) \dd x + \frac{1}{2} \langle B^M, B^M \rangle_t\\
  & = \int_0^t B_s \otimes \dd B_s + \frac{1}{2} \langle B, B \rangle_t + t
  \int \Phi (x) \otimes \mathcal{L}_A \Phi (x) \dd x\\
  & = \int_0^t B_s \otimes \circ \dd B_s - t \int \nabla \Phi^i \cdot (a^A \nabla \Phi^j) \dd x.
\end{align*}

\section{Proof of Proposition \ref{prop:area-variation}}\label{sec:proof-area-variation}

We write $\| f \|_{p, [s, t]}$ for the $p$-variation of $f$ restricted to the
interval $[s, t]$.

\begin{definition}
  A {\emph{control function}} is a map $c : \Delta_T \to [0, \infty)$ with $c
  (t, t) = 0$ for all $t \in [0, T]$ and such that $c (s, u) + c (u, t) \leqslant       c
  (s, t)$ for all $0 \leqslant       s \leqslant       u \leqslant       t \leqslant       T$.
\end{definition}

Observe that if $f : [0, T] \rightarrow \mathbb{R}^d$ satisfies $|f_{s, t} |^p
\leqslant       c (s, t)$ for all $(s, t) \in \Delta_T$, then the $p$-variation of $f$ is
bounded from above by $c (0, T)^{1 / p}$. Indeed, we have for any partition
$\pi$ of $[0, T]$
\[ \left( \sum_{[s, t] \in \pi} | f_{s, t} |^p \right)^{1 / p} \leqslant
   \left( \sum_{[s, t] \in \pi} c (s, t) \right)^{1 / p} \leqslant c (0, T)^{1
   / p} . \]
Conversely, if $f$ is of finite $p$-variation, then $c (s, t) \assign \| f
\|_{p, [s, t]}^p$ defines a control function because $c (t, t) = | f_{t, t} |
= 0$ and
\begin{align*}
  c (s, u) + c (u, t) & = \sup_{\pi \text{ Part. of } [s, u]} \sum_{[r,
  v] \in \pi} | f_{r, v} |^p + \sup_{\pi \text{ Part. of } [u, t]}
  \sum_{[r, v] \in \pi} | f_{r, v} |^p\\
  & = \sup_{\tmscript{\begin{array}{c}
    \pi \text{ Part. of }  [s, t]\\
    \text{s.t. } u \in \pi
  \end{array}}} \sum_{[r, v] \in \pi} | f_{r, v} |^p \leqslant \sup_{\pi
  \text{ Part. of } [s, t]} \sum_{[r, v] \in \pi} | f_{r, v} |^p \\
  &= c(s,t) .
\end{align*}

Note also that the sum of two control functions is a control function.
Proposition \ref{prop:area-variation} directly follows from the next result:

\begin{proposition}\label{prop:area-variation-precise}
  Let $(Y_t)_{t \in
  [0, T]}$ be a c\`adl\`ag adapted process such that $\| Y \|_{p , [0, T]} <
  \infty$ and let $N$ be a square-integrable martingale. Set $A_{s, t} \assign \int_s^t Y_{r -} d N_r - Y_{s } N_{s, t}$. Then we have for all $p,q > 2$ and all $r > \left( \frac{1}{p} + \frac{1}{q}
  \right)^{- 1}$:
  \begin{equation}\label{eq:area-variation-1}
    \| A \|_{r , [0, T]} \lesssim \left(1 + | \log \| Y \|_{p , [0, T]} |\right) \| Y
    \|_{p , [0, T]} \left( K^{\frac{1}{q}} + \| N \|_{q , [0, T]} \right),
  \end{equation}
  where $K$ is a random variable with $\mathbb{E} [K^{\frac{2}{q}}] \lesssim
  \mathbb{E} [\langle N \rangle_T]$. In particular, we get for $\varepsilon >
  0$
  \begin{equation}\label{eq:area-variation-2}
  	\mathbb{E} \left[\| A \|_{r , [0, T]}^{1 - \varepsilon}\right] \lesssim \left(1
     +\mathbb{E} \left[\| Y \|_{p , [0, T]}^2\right]^{1 / 2}\right) \mathbb{E} [\langle N
     \rangle_T]^{1 / 2} .
  \end{equation}

\end{proposition}

\begin{remark}
  For $p < 2$ it follows directly from Young integration estimates that $\| A
  \|_{1 / (1 / p + 1 / r), [0, T]} \lesssim \| Y \|_{p, [0, T]} \| N \|_{r,
  [0, T]}$ whenever $r > 2$ is such that $1 / p + 1 / r > 1$.
\end{remark}

\begin{proof}
  Define the stopping times $\tau^n_0 := 0$ and $\tau^n_{k + 1} \assign \inf \{ t \geqslant
  \tau^n_k : | Y_{\tau^n_k, t} | \geqslant 2^{- n} \}$ and set
  \[ Y^n_t := \sum_{k = 0}^{\infty} \mathds{1}_{(\tau^n_k, \tau^n_{k +
     1}]} (t) Y_{\tau^n_k}, \]
  such that $\sup_{t \in [0, T]} | Y_{t -} - Y^n_t | \leqslant 2^{- n}$, where
  $Y_{t -} \assign \lim_{s \uparrow t} Y_s$ and $Y_{0 -} : = Y_0$ and we also
  write $(Y_-)_t \assign Y_{t -}$. We have
  \begin{equation}\label{eq:area-variation-pr1} | A_{s, t} | \leqslant \left| \int_s^t (Y_{r -} -
    Y^n_r) d N_r \right| + \left| \int_s^t Y^n_r d N_r - Y_{s } N_{s, t}
    \right| .
  \end{equation}
  The first term on the right hand side is bounded for $q > 2$ and $n \in
  \mathbb{Z} \setminus \{ 0 \}$ by
  \begin{align}\label{eq:area-variation-pr2}
    \left| \int_s^t (Y_{r -} - Y^n_r) d N_r \right| 
     & \leqslant | n | 2^{- n} c (s, t)^{\frac{1}{q}}
    K^{\frac{1}{q}},
  \end{align}
  where we define
  \begin{equation}\label{eq:area-variation-pr3}
    K \assign \sum_{m \in \mathbb{Z} \setminus \{ 0 \}} | m |^{- q} 2^{m q}
    \left\| \int_0^{\cdummy} (Y_{r -} - Y^m_r) d N_r \right\|_{q , [0, T]}^q
  \end{equation}
  and
  \begin{equation}\label{eq:area-variation-pr4}
    c (s, t) \assign \frac{\sum_{m \in \mathbb{Z} \setminus \{ 0 \}} | m |^{-
    q} 2^{m q} \left\| \int_0^{\cdummy} (Y_{r -} - Y^m_r) d N_r \right\|_{q ,
    [s, t]}^q}{\sum_{m \in \mathbb{Z} \setminus \{ 0 \}} | m |^{- q} 2^{m q}
    \left\| \int_0^{\cdummy} (Y_{r -} - Y^m_r) d N_r \right\|_{q , [0, T]}^q}
    + \frac{\| Y_- \|_{p , [s, t]}^p + \| Y \|_{p , [s, t]}^p}{\| Y \|_{p ,
    [0, T]}^p} + \frac{\| N \|_{q , [s, t]}^q}{\| N \|_{q , [0, T]}^q} .
  \end{equation}
  Note that $\| Y \|_{p , [0, T]}^p = \| Y_- \|_{p , [0, T]}^p$, and therefore
  $c (s, t) \leqslant c (0, T) = 4$.

  To bound the second term in \eqref{eq:area-variation-pr1} let $t_0 \assign \min \{
  \tau^n_k : \tau^n_k \in (s, t) \} \wedge t$. If $t_0 = \tau^n_{k_0} < t$, we
  let $\tau^n_{k_0 + m - 1}$ be the maximal $\tau^n_k \in (s, t)$, for $m
  \geqslant 1$, and we write $t_k \assign \tau^n_{k_0 + k}$ for $k = 1, \ldots, m
  - 1$, while $t_m \assign t$. Otherwise we set $m \assign 0$. Then
  \begin{align}\label{eq:area-variation-pr5} \nonumber
    \left| \int_s^t Y^n_r d N_r - Y_{s } N_{s, t} \right| \leqslant & \left|
    \int_s^{t_0} Y^n_r d N_r - Y_{s } N_{s, t_0} \right| + \left|
    \int_{t_0}^{t_m} (Y^n_r - Y_{t_0}^n) d N_r \right|\\
     & + | (Y_{t_0}^n - Y_{t_0 }) N_{t_0, t_m} | + |
    (Y_{t_0 } - Y_{s }) N_{t_0, t_m} |,
  \end{align}
  The first and
  third term on the right hand side are bounded by
  \begin{align}\label{eq:area-variation-pr6} \nonumber
    \left| \int_s^{t_0} Y^n_r d N_r - Y_{s } N_{s, t_0} \right| + |
    (Y_{t_0}^n - Y_{t_0 }) N_{t_0, t_m} | & = | Y^n_{s } N_{s, t_0} - Y_{s
    } N_{s, t_0} | + | (Y_{t_0}^n - Y_{t_0 }) N_{t_0, t_m} |\\
    & \leqslant 2 \times 2^{- n} c (s, t)^{\frac{1}{q}} \| N \|_{q , [0, T]},
  \end{align}
  and the last term is controlled by
  \[ | (Y_{t_0} - Y_{s}) N_{t_0, t_m} | \leqslant c (s, t)^{\frac{1}{p} +
     \frac{1}{q}} \| Y \|_{p , [0, T]} \| N \|_{q , [0, T]} . \]
  To bound the second term in~{\eqref{eq:area-variation-pr5}} we use an idea from \cite[Theorem~4.12]{Perkowski2016}: We apply Young's maximal inequality despite
  the fact that $Y_-$ and $N$ are not sufficiently regular for the construction of the Young integral. This
  will give us a divergent factor in $n$, but on the other hand it gives us a large power of $c (s, t)$. Then we
  balance this term with the other terms in the upper bound for $| A_{s, t} |$
  (which all contain a factor $2^{- n}$) by choosing the right $n$. Young's idea
  is to successively delete points from the partition $t_0 < \cdots <
  t_m$ in order to pass from $\sum_{k = 0}^{m - 1} Y_{t_k } N_{t_k, t_{k +
  1}}$ to $Y_{t_0 } N_{t_0, t_m}$. We want to delete the points in an optimal
  way, and to express what optimal means we first renormalize $Y$ and $N$:
  \[ \int_{t_0}^{t_m} (Y^n_r - Y_{t_0}^n) d N_r = \sum_{k = 0}^{m - 1} (Y_{t_k
     } - Y_{t_0 }) N_{t_k, t_{k + 1}} = \sum_{k = 0}^{m - 1} (\tilde{Y}_{t_k
     } - \tilde{Y}_{t_0 }) \tilde{N}_{t_k, t_{k + 1}} \| Y \|_{p , [0, T]}
     \| N \|_{q , [0, T]}, \]
  where $\tilde{Y} = \frac{Y}{\| Y \|_{p , [0, T]}}$ and $\tilde{N} =
  \frac{N}{\| N \|_{q , [0, T]}}$. Then $c$ controls  $\tilde{Y}$ and $\tilde{N}$, and by the superadditivity
  of $c$ there exists $\ell \in \{ 1, \ldots, m - 1 \}$ with $c
  (t_{\ell - 1}, t_{\ell + 1}) \leqslant \frac{2}{m - 1} c (s, t)$ whenever $m
  > 1$ (for $m = 1$ the integral vanishes). By deleting the point $t_{\ell}$
  from the partition and subtracting the resulting expression, we get
  \begin{align*}
    | \tilde{Y}_{t_{\ell - 1} } \tilde{N}_{t_{\ell - 1}, t_{\ell}} +
    \tilde{Y}_{t_{\ell} } \tilde{N}_{t_{\ell}, t_{\ell + 1}} -
    \tilde{Y}_{t_{\ell - 1} } \tilde{N}_{t_{\ell - 1}, t_{\ell + 1}} | & = |
    \tilde{Y}_{t_{\ell - 1} , t_{\ell} } \tilde{N}_{t_{\ell}, t_{\ell + 1}}
    | \leqslant c (t_{\ell - 1}, t_{\ell + 1})^{\frac{1}{p} + \frac{1}{q}}\\
    & \leqslant \left( \frac{2}{m - 1} c (s, t) \right)^{\frac{1}{p} +
    \frac{1}{q}} .
  \end{align*}
  We proceed by successively deleting all points except $t_0$ and $t_m$ from
  the partition, each time in such an ``optimal'' way, and obtain
  \[ \left| \sum_{k = 0}^{m - 1} (\tilde{Y}_{t_k } - \tilde{Y}_{t_0 })
     \tilde{N}_{t_k, t_{k + 1}} \right| \leqslant \sum_{k = 1}^{m - 1} \left(
     \frac{2}{k} c (s, t) \right)^{\frac{1}{p} + \frac{1}{q}} \lesssim (m -
     1)^{1 - \frac{1}{p} - \frac{1}{q}} c (s, t)^{\frac{1}{p} + \frac{1}{q}} .
  \]
  Moreover,
  \[ m - 1 =\# \{ k : \tau^n_k \in (\tau^n_{k_0}, t) \} \leqslant 2^{n p} \| Y
     \|_{p , [s, t]}^p \leqslant 2^{n p} c (s, t) \| Y \|_{p , [0, T]}^p . \]
  So overall
  \begin{equation}\label{eq:area-variation-pr7}
    \left| \int_{t_0}^{t_m} (Y^n_r - Y_{t_0}^n) d N_r \right| \lesssim 2^{n p
    \left( 1 - \frac{1}{p} - \frac{1}{q} \right)} c (s, t) \| Y \|_{p , [0,
    T]}^{p \left( 1 - \frac{1}{p} - \frac{1}{q} \right) + 1} \| N \|_{q , [0,
    T]} .
  \end{equation}
  We combine \eqref{eq:area-variation-pr1}, \eqref{eq:area-variation-pr2}, \eqref{eq:area-variation-pr5}, \eqref{eq:area-variation-pr6}, \eqref{eq:area-variation-pr7},  and obtain the key
  bound
  \begin{align*}
    | A_{s, t} | \lesssim & | n | 2^{- n} c (s, t)^{\frac{1}{q}} \left(
    K^{\frac{1}{q}} + \| N \|_{q , [0, T]} \right) + 2^{n p \left( 1 -
    \frac{1}{p} - \frac{1}{q} \right)} c (s, t) \| Y \|_{p , [0, T]}^{p \left(
    1 - \frac{1}{p} - \frac{1}{q} \right) + 1} \| N \|_{q , [0, T]}\\
    & + c (s, t)^{\frac{1}{p} + \frac{1}{q}} \| Y \|_{p , [0, T]} \| N \|_{q
    , [0, T]} .
  \end{align*}
  To balance the first and
  second term, choose $n \in \mathbb{Z} \setminus \{ 0 \}$ so that $\frac{1}{2} < 2^{n} c (s, t)^{\frac{1}{p}} \| Y \|_{p , [0,T]}\leqslant 2$. Then
  \begin{align*}
    | A_{s, t} | & \lesssim | n | c (s, t)^{\frac{1}{p} + \frac{1}{q}} \| Y
    \|_{p , [0, T]} \left( K^{\frac{1}{q}} + \| N \|_{q , [0, T]} \right)\\
    & \simeq \left| \log \left( c (s, t)^{\frac{1}{p}} \| Y \|_{p , [0, T]}
    \right) \right| c (s, t)^{\frac{1}{p} + \frac{1}{q}} \| Y \|_{p , [0, T]}
    \left( K^{\frac{1}{q}} + \| N \|_{q , [0, T]} \right)\\
    & \leqslant \left( \left| \log \left( c (s, t)^{\frac{1}{p}} \right)
    \right| + | \log \| Y \|_{p , [0, T]} | \right) c (s, t)^{\frac{1}{p} +
    \frac{1}{q}} \| Y \|_{p , [0, T]} \left( K^{\frac{1}{q}} + \| N \|_{q ,
    [0, T]} \right)
  \end{align*}
  and since $c (s, t)^{\frac{1}{p}} \leqslant 4$ we have $\left| \log \left( c
  (s, t)^{\frac{1}{p}} \right) \right| \lesssim c (s, t)^{- \varepsilon}$ for
  $\varepsilon > 0$. Thus, we get for all $r > \left( 1/p +
  1/q \right)^{- 1}$ our first inequality~\eqref{eq:area-variation-1}. This yields
  \begin{align*}
    \mathbb{E} [\| A \|_{r , [0, T]}^{1 - \varepsilon}] & \lesssim \mathbb{E}
    \big[\big((1 + | \log \| Y \|_{p , [0, T]} |) \| Y \|_{p , [0, T]}\big)^{2 - 2
    \varepsilon}\big]^{\frac{1}{2}} \mathbb{E} \left[ K^{\frac{2}{q}} + \| N \|_{q
    , [0, T]}^2 \right]^{\frac{2 - 2 \varepsilon}{2}} .
  \end{align*}
  The second expectation on the right hand side is easy to control: since
  $2/q < 1$ we have $\left( \sum_m a_m \right)^{2/q} \leqslant
  \sum_m a_m^{2/q}$, and Theorem \ref{thm:lepingle BDG} (L\'epingle's $q$-variation Burkholder-Davis-Gundy inequality), we get
  \begin{align*}
    \mathbb{E} \left[ K^{\frac{2}{q}} + \| N \|_{q , [0, T]}^2 \right] &
    \leqslant \sum_{m \in \mathbb{Z} \setminus \{ 0 \}} | m |^{- 2} 2^{2 m}
    \mathbb{E} \left[ \left\| \int_0^{\cdummy} (Y_{r -} - Y^m_r) d N_r
    \right\|_{q , [0, T]}^2 \right] +\mathbb{E} [\| N \|_{q , [0, T]}^2]\\
    & \lesssim \sum_{m \in \mathbb{Z} \setminus \{ 0 \}} | m |^{- 2} 2^{2 m}
    \mathbb{E} \left[ \left[ \int_0^{\cdummy} (Y_{r -} - Y^m_r) d N_r
    \right]_T \right] +\mathbb{E} [[N]_T]\\
    & = \sum_{m \in \mathbb{Z} \setminus \{ 0 \}} | m |^{- 2} 2^{2 m}
    \mathbb{E} \left[_{} \left\langle \int_0^{\cdummy} (Y_{r -} - Y^m_r) d N_r
    \right\rangle_T \right] +\mathbb{E} [\langle N \rangle_T]\\
    & = \sum_{m \in \mathbb{Z} \setminus \{ 0 \}} | m |^{- 2} 2^{2 m}
    \mathbb{E} \left[_{} \int_0^T (Y_{r -} - Y^m_r)^2 d \langle N \rangle_r
    \right] +\mathbb{E} [\langle N \rangle_T]\\
    & \lesssim \mathbb{E} [\langle N \rangle_T] .
  \end{align*}
  The remaining expectation is bounded by
  \[ \mathbb{E} [((1 + | \log \| Y \|_{p , [0, T]} |) \| Y \|_{p , [0, T]})^{2
     - 2 \varepsilon}] \lesssim 1 +\mathbb{E} [\| Y \|_{p , [0, T]}^2], \]
  and this concludes the proof.
\end{proof}

\appendix\section{Auxiliary estimates}

\begin{lemma}[Iterated Kipnis-Varadhan estimate]
  \label{lem:iterated-kv}
  Let $H \in \mathcal{H}^{- 1} \cap L^2 (\pi)$ and let
  $A$ be a continuous adapted process of finite variation. Then
  \[ \mathbb{E} \left[ \sup_{t \leqslant T} \left| \int_0^t A_s H (X_s) \dd
     s \right| \right] \lesssim \mathbb{E} [\sup_{t \leqslant T} | A_t |^2]^{1
     / 2} T^{1 / 2} \| H \|_{- 1}, \]
  so in particular we get for $A_t = \int_0^t G (X_s) \dd s$ with $G \in
  \mathcal{H}^{- 1} \cap L^2 (\pi)$
  \[ \mathbb{E} \left[ \sup_{t \leqslant T} \left| \int_0^t \int_0^s G (X_r)
     \dd r H (X_s) \dd s \right| \right] \lesssim T \| G \|_{- 1} \| H
     \|_{- 1} . \]
\end{lemma}

\begin{proof}
  The second inequality follows from the first one together with the usual
  Kipnis-Varadhan estimate from Lemma~\ref{lem:kv-p-var}. To show the first
  inequality, let $\Psi \in \mathcal{C}$ and apply
  Lemma~\ref{lem:fb-decomp}:
\begin{align*}
    \int_0^t A_s H (X_s) \dd s & = \frac{1}{2} \int_0^t A_s \dd M^\Psi_s -
    \frac{1}{2} \int_{T - t}^T (A_T - A_{T - s}) \dd r \dd \hat{M}^\Psi_s\\
    & \quad + \frac{1}{2} A_T (\hat{M}^\Psi_T - \hat{M}^\Psi_{T - t}) + \int_0^t
    A_s (H (X_s) -\mathcal{L}_S \Psi (X_s)) \dd s,
  \end{align*}
  where we need that $A$ is continuous and of finite variation in order to
  interpret the integrals against $\hat{M}^\Psi$ in a pathwise sense and without
  having to worry about the difference of forward and backward integral. Now
  we get from the Burkholder-Davis-Gundy and Cauchy-Schwartz inequalities
  together with Lemma~\ref{lem:kv-p-var}
\begin{align*}
    \mathbb{E} \left[ \sup_{t \leqslant T} \left| \int_0^t A_s H (X_s) \dd
    s \right| \right] & \lesssim \mathbb{E} [\sup_{t \leqslant T} | A_t
    |^2]^{1 / 2} T^{1 / 2} \| \Psi \|_1\\
    & \quad +\mathbb{E} \left[ \sup_{t \leqslant T} \left| \int_0^t A_s (H
    (X_s) -\mathcal{L}_S \Psi (X_s)) \dd s \right| \right]\\
    & \lesssim \mathbb{E} [\sup_{t \leqslant T} | A_t |^2]^{1 / 2} (T^{1 / 2}
    \| \Psi \|_1 + T \| H -\mathcal{L}_S \Psi \|_{L^2 (\pi)}) .
  \end{align*}
  By approximation we can take $\Psi = \Phi^H_{\lambda}$ as the solution to
  the Poisson equation $(\lambda -\mathcal{L}_S) \Phi^H_{\lambda} = - H$ and
  as in the proof of Corollary~\ref{cor:p-var-iterated-kv} we use that $\|
  \Phi^H_{\lambda} \|_1 \leqslant \| H \|_{- 1}$ for all $\lambda > 0$ and
  that $\| H -\mathcal{L}_S \Phi^H_{\lambda} \|_{L^2 (\pi)} \rightarrow 0$ as
  $\lambda \rightarrow 0$ to deduce the claimed estimate.
\end{proof}

The following is the L\'epingle p-variation inequality \cite[Proposition~2]{Lepingle1975} which is here commonly combined with the well-known Burkholder-Davis-Gundy inequality.
\begin{theorem}[L\'epingle $p$-variation Burkholder-Davis-Gundy inequality]\label{thm:lepingle BDG}
Let $(M_t)_{t \geqslant 0}$ be a local martingale with trajectories in $D (\mathbb{R}_+, \mathbb{R}^m)$. For every $T>0$ and $p>2$
\[
c_p \bbE \left[ [M]_T \right] \leqslant  \bbE \left[ \| {M} \|_{p , [0, T]}^2 \right]
\leqslant
C_p \bbE \left[ [M]_T \right],
\]
where $c_p,C_p>0$.
\end{theorem}

Next lemma is a strengthening of the ergodic theorem to give a path uniform convergence.
\begin{lemma}\label{lem:stationary-uniform} Let $(Y_t)_{t \geqslant 0}$ be a process with trajectories in $D (\mathbb{R}_+, \mathbb{R}^m)$ and with stationary
  increments and such that $\mathbb{E} [\sup_{t \in [0, T]} | Y_t |] \leqslant C T$ for all $T > 0$ and
  such that $n^{- 1} Y_n \rightarrow a$ for some $a \in \mathbb{R}^m$, both
  a.s. and in $L^1$. {Assume also that $(\sup_{t \in [0,T]} n^{-1} |Y_{nt}|)_{n \in \bb N}$ is uniformly integrable for all $T>0$.} Then we have for all $T > 0$
  \[ \lim_{n \rightarrow \infty} \mathbb{E} \left[\sup_{t \leqslant T} | n^{- 1}
     Y_{n t} - a t |\right] = 0. \]
\end{lemma}

\begin{proof}
  This follows from a minor adaptation of the proof of Theorem~2.29 in
  {\cite{Komorowski2012}}: Like in that proof we decompose
\begin{align*}
    | n^{- 1} Y_{n t} - a t | & \leqslant \sup_{s \in [0, 1]} \frac{|
    Y_{\lfloor n t \rfloor + s} - Y_{\lfloor n t \rfloor} |}{n} +
    \frac{\lfloor n t \rfloor}{n} \left| \frac{Y_{\lfloor n t
    \rfloor}}{\lfloor n t \rfloor} - a \right| + | a | \left( t -
    \frac{\lfloor n t \rfloor}{n} \right) .
  \end{align*}
  The last term on the right hand side is bounded by $| a | / n$. The first
  term on the right hand side is bounded for all $t \in [0, T]$ by
  \[ \sup_{s \in [0, 1]} \frac{| Y_{\lfloor n t \rfloor + s} - Y_{\lfloor n t
     \rfloor} |}{n} \leqslant T \max_{k \leqslant \lfloor n T \rfloor}
     \frac{\sup_{s \in [0, 1]} | Y_{k + s} - Y_k |}{\lfloor n T \rfloor}, \]
  and by Lemma~2.30 in {\cite{Komorowski2012}} the right hand side vanishes as
  $n \rightarrow \infty$, both a.s. and in $L^1$ {(here we need that $Y$ has stationary increments)}.
  {To handle the last remaining term, we decompose for $\delta \in (0,T)$:
  \[
  	\sup_{t \in [0,T]} \frac{\lfloor n t \rfloor}{n} \left| \frac{Y_{\lfloor n t \rfloor}}{\lfloor n t \rfloor} - a \right| = \max_{k \le \lfloor n T \rfloor} \frac{k}{n} \left| \frac{Y_{k}}{k} - a \right| \le \max_{k \le \lfloor n \delta \rfloor} \frac{k}{n} \left| \frac{Y_{k}}{k} - a \right| + \max_{\lfloor n \delta \rfloor < k \le \lfloor n T \rfloor} \frac{k}{n} \left| \frac{Y_{k}}{k} - a \right|.
  \]
  The second term on the right hand side converges almost surely to zero by assumption, and it is bounded from above by $\sup_{t \in [0,T]} n^{-1} |Y_{nt}|$. Since $(\sup_{t \in [0,T]} n^{-1} |Y_{nt}|)_n$ is uniformly integrable by assumption, this second term also converges to zero in $L^1$. The remaining part satisfies
  \[
  	\bb E \left[ \max_{k \le \lfloor n \delta \rfloor} \frac{k}{n} \left| \frac{Y_{k}}{k} - a \right| \right] \le \frac1n \bb E \left[\max_{k \le \lfloor n \delta \rfloor}|Y_k|\right] + |a|\delta \le C \delta + |a|\delta,
  \]
  where the last part follows by assumption on $Y$. The proof is then completed by sending $\delta \to 0$.
  }
\end{proof}

\bibliography{all}
\bibliographystyle{alpha}

\end{document}